\documentclass{amsart}
\usepackage {amssymb}
\usepackage {amsmath}
\usepackage{amsthm}
\usepackage{graphicx}
\usepackage {amscd}
\usepackage {epic}
\usepackage{tikz,xcolor}
\usepackage[alphabetic]{amsrefs}
\usepackage[colorlinks, citecolor=blue]{hyperref}

\newcommand{\Lk}[0]{\operatorname{Lk}}  
\newcommand{\lk}[0]{\operatorname{Lk}}  
\newcommand{\st}[0]{\operatorname{St}}  
\newcommand{\cst}[0]{\overline{\operatorname{St}}}  
 
\newcommand{\ex}[0]{\operatorname{Ex}}  
\newcommand{\spec}[0]{\operatorname{Spec}} 
\newcommand{\red}[0]{\operatorname{red}} 
\newcommand{\supp}[0]{\operatorname{Supp}} 
\newcommand{\map}{\dasharrow}
\newcommand{\cone}[0]{\operatorname{Cone}} 

\newcommand{\DMR}{{\mathcal{DMR}}}
\newcommand{\DR}{{\mathcal{DR}}}
\newcommand{\D}{{\mathcal{D}}}

\newtheorem{thm}{Theorem}

\newtheorem{lem}[thm]{Lemma}
\newtheorem{cor}[thm]{Corollary}

\newtheorem{prop}[thm]{Proposition}

\theoremstyle{definition}
\newtheorem{defn}[thm]{Definition}

\newtheorem{say}[thm]{}
\newtheorem{exmp}[thm]{Example}

\newtheorem{rem}[thm]{Remark}          

\newtheorem*{ack}{Acknowledgments}      

\newtheorem{defn-thm}[thm]{Definition--Theorem}  
\newtheorem{defn-lem}[thm]{Definition--Lemma}  

\newtheorem{comp}[thm]{Complement} 

\theoremstyle{remark}

\input{xy}
\xyoption{all}

\begin{document}
\title{The dual complex of  singularities}

\author{Tommaso de Fernex, J\'anos Koll\'ar and Chenyang Xu}

\date{\today}

\maketitle{}
\tableofcontents
\section{Introduction}

To every simple normal crossing variety $E$
one can associate a cell complex $\D(E)$, called \emph{the dual complex} of $E$,
that describes how the irreducible components of $E$ intersect;
see Definition \ref{dual.cpx.defn} for details. 
Over the 
complex numbers,  one can think of $\D(E)$ as the 
\emph{combinatorial part} of the topology of $E$.

Let $(0\in X)$ be a (not necessarily isolated) singularity
 and   $f\colon Y\to X$   a  resolution such that 
$E:=\supp\bigl(f^{-1}(0)\bigr)$ is a 
simple normal crossing divisor.
The corresponding dual complex $\D(E)$ depends on the choice of $Y$
but, as   proved in increasingly general forms
in the papers \cites{Stepanov06,Thuillier07,Stepanov07, ABW11, Payne11},
 the different $\D(E)$ are all homotopy equivalent,
 even  simple-homotopy equivalent; see
Definition \ref{collapse.defn}. Their
(simple-)homotopy equivalence class is denoted by $\DR(0\in X)$.

The main result of this paper is that in many cases, for instance
for isolated singularities, one can do even better and 
select a ``minimal'' representative  $\DMR(0\in X)$
that is well defined up-to piecewise linear homeomorphism.

For surfaces, this representative is given by
$\D(E)$ of any  resolution $f:Y\to X$ such that
$E:=\supp\bigl(f^{-1}(0)\bigr)$ is a 
nodal curve with smooth irreducible components and
$K_{Y}+E$ is $f$-nef, that is, $\bigl((K_{Y}+E)\cdot C\bigr)\geq 0$
for every $f$-exceptional curve $C$.

In higher dimensions such resolutions usually do not exist
but the Minimal Model Program tells us that such objects do exist
if we allow $Y$ and $E$ to be mildly singular.
The relevant notion is  {\it divisorial log terminal} or {\it dlt,}
see Definition \ref{nor.conv.say}. 
Every simple normal crossing pair $(Y,E)$ is also dlt and
for general dlt pairs  $(Y,E)$ one can define the dual complex $\D(E)$
by simply ignoring the singularities, see Definition \ref{dual.cpx.defn}.
Furthermore, for every normal variety $X$ for which
$K_X$ is $\mathbb Q$-Cartier, there are
(usually infintely many) projective, birational  morphisms  $p:Y\to X$ 
such that $(Y,E)$ is dlt and $K_Y+E$ is $p$-nef, where $E$ 
is the divisorial part of the exceptional set $\ex(p)$; see \cite{OX12}.
These are called
the {\it dlt modifications} of $X$, see Definition \ref{dlt-mod.defn}.
(Conjecturally, dlt modifications  exist for
any $X$.) 

 Our main theorem is the following.
Here we state it for isolated singularities; the general form is
given in Section \ref{pf.sect}.

\begin{thm}\label{DMR-thm-isolated}
Let $(0\in X)$ be an isolated singularity over ${\mathbb C}$
such that $K_X$ is $\mathbb Q$-Cartier.
Let $p:Y\to X$ be a proper, birational morphism 
and $E\subset Y$ the divisorial part of  $\supp p^{-1}(0)$.
Assume that  $(Y,E)$ is dlt and $E\neq \emptyset$.    
\begin{enumerate}
\item If $p$ is a dlt modification, that is if $K_Y+E$ is $p$-nef, then
the dual complex $\D(E)$ is independent of $p$, up-to
PL homeomorphism. This defines a PL homeomorphism class 
associated to $(0\in X)$; we
 denote it by $\DMR(0\in X)$.
\item If  $E=\supp p^{-1}(0)$ then
 the dual complex $\D(E)$ is
 simple-homotopy equivalent to
$\DMR(0\in X)$.
\item If  $X$ is $\mathbb Q$-factorial, $p$ is projective 
and $p$ is  an isomorphism over $X\setminus \{0\}$   
then  $\D(E)$ collapses to $\DMR(0\in X)$.
\end{enumerate}
\end{thm}

Here $\DMR$ stands for the ``Dual complex of the Minimal
divisorial log terminal partial Resolution''
and a {\it collapse}
 is a particularly simple type of homotopy equivalence;
 see Definition   \ref{collapse.defn}.

The assumptions in (\ref{DMR-thm-isolated}.1--2) are most likely
optimal. 

If the minimal model conjecture holds for dlt pairs then
 $\DMR(0\in X)$ is defined even if  $K_X$ is not $\mathbb Q$-Cartier,
but it need not be homotopy equivalent to  $\DR(0\in X)$.
On the other hand, even in these cases, 
we can apply the more general Theorem \ref{DMR-thm}
to  a pair $(X,\Delta)$ where $\Delta\sim_{\mathbb Q} -K_X$ is a general 
$\mathbb Q$-divisor.

Even for hypersurface singularities the various
$\D(E)$ are not homeomorphic to each other,
even  the number of connected components of $\D(E)$ can change.

By contrast, the necessity of the 
assumptions in (\ref{DMR-thm-isolated}.3) is not clear and
we do not have any examples where $\D(E)$ does not collapse to $\DMR(0\in X)$.
Note also that many singularities are $\mathbb Q$-factorial 
(for instance all isolated  complete intersection singularities
 of dimension $\geq 4$) and many can be made
$\mathbb Q$-factorial  by choosing a suitable algebraic model
\cite{par-sri}.

 If $X$ is  log terminal    
 then its dlt modification is the identity
map $X\cong X$, thus $E=\emptyset$. Nonetheless, 
in this case we  define $\DMR(0\in X)$ to be a point,
rather than the empty set.
By a simple trick of introducing an auxiliary divisor,
we can apply the general version of 
Theorem \ref{DMR-thm-isolated} to this case and 
prove the following.

\begin{thm}\label{c-dlt} 
Let $0\in X$ be a point on  a dlt pair $(X,\Delta)$ over ${\mathbb C}$. 
Then $\DR(0\in X)$ is contractible.
\end{thm}

Although the connection is not immediate, 
we derive Theorems \ref{DMR-thm-isolated}--\ref{c-dlt} from the following more
global result which is the  
main technical theorem of the paper.
It asserts that for a given dlt pair $(X,\Delta)$,
one should focus on the sum of those divisors that appear in $\Delta$ 
with coefficient 1. This divisor,
denoted by $\Delta^{=1}$, is also the union of all
log canonical centers of $(X,\Delta)$.

\begin{thm}\label{t-sdr}
Let $(X,\Delta)$ be a  quasi projective, dlt pair  over ${\mathbb C}$
and $g\colon Y\to X$  a projective resolution 
such that $E:=\supp g^{-1} (\Delta^{=1})$
 is a simple normal crossing divisor. Then 
\begin{enumerate}
\item The dual complex  $\D(E)$
is simple-homotopy equivalent to $\D(\Delta^{=1})$.
\item If $E$ is $\mathbb Q$-Cartier then  
$\D(E)$ collapses to $\D(\Delta^{=1})$.
\end{enumerate}
\end{thm}

For a  simple normal crossing divisor
$ E\subset X$ the intersections of the various $E_i\subset E$
 are also the log canonical centers of the pair
$(X,E)$, thus one can also view the  dual complex
as describing the combinatorics of log canonical centers. 

For a general log canonical pair $(X,\Delta)$ these two approaches give
different objects and it is of interest to study the cases
when they are the same. This leads to the 
concept of quotient-dlt pairs, see Definition \ref{qdlt.defn}.

Quotient-dlt pairs constitute a useful subclass of 
log canonical pairs which is preserved by
Fano contractions. In  Section \ref{sec.quot.dlt} 
this leads to a  short proof of
the following, which extends  earlier results of 
\cite{Kollar07, HX09}.

\begin{thm}\label{t-RC}
Let $f\colon X\to (0\in C)$ be a flat projective morphism to a
 smooth, pointed  curve over ${\mathbb C}$.
 Assume that the general fibers $X_t$ are smooth, rationally connected
and $F_0:=\red f^{-1}(0)$ is a simple normal crossing divisor. Then 
  $\D(F_0)$ is contractible. 
\end{thm}

\begin{say}[Topological remarks] 
One should think of  elementary collapses and their inverses 
(see Definition \ref{collapse.defn}) as the
 obvious homotopy equivalences.
A sequence of them gives simple-homotopy equivalences.

By a theorem of Whitehead, every homotopy equivalence
between simply connected simplicial complexes 
is homotopic to a simple-homotopy equivalence, but
this can fail when $\pi_1\neq 1$; see \cite{Coh73}.

Being collapsible is much stronger than being
contractible. For instance, the {\it dunce hat} or
{\it Bing's house with two rooms} are
 contractible, even simple-homotopy equivalent to a point,
yet they are not collapsible; cf. \cite{zee61}.

 We do not have any examples 
as in Theorems \ref{c-dlt}   or \ref{t-RC}
where the dual complex is not collapsible.
There are some indications, for example
 Theorems \ref{c-dlt-gal}--\ref{t-RC-gal}, that dual complexes
coming from algebraic geometry are somewhat special. 
\end{say}

{\it History.} The dual graph of the exceptional curves of 
resolutions of surface singularities has been studied  for a long time.
To the best of our knowledge, higher dimensional versions first
appeared in \cite{Kul77} for degenerations of K3 surfaces.
The members of the seminar \cite{FM83} essentially knew Theorem \ref{c-dlt} 
 for 3-folds and Theorem \ref{t-RC} for surfaces.

The connectedness conjecture of \cite{Sho92}, 
proved in \cite[Sec.17]{K-etal92},
was the first indication that exceptional divisors with
discrepancy $\leq -1$ play a special role.
The role of the dual complex of these divisors 
is studied in \cite{Fuj-c2, Fuj-c1}, especially
for lc singularities. The effect of the MMP on the dual complex
is studied in \cite{FujTak-c3} which 
 essentially leads to the proof of  Theorem~\ref{DMR-thm-isolated} 
for log canonical $X$. 

Much of the early work on the abundance conjecture
involves understanding how a dual complex changes under birational
maps. In retrospect, versions of  Proposition~\ref{crep.bir.dlt.models} 
are contained  in \cite{K-etal92, MR1284817, Fuj-c4}.

Theorem \ref{c-dlt}  is obvious if $\dim X=2$ 
  and the simple connectedness of $\DR(0\in X)$ 
follows from \cite{Koll93,Takayama03}. Much of 
the 3-dimensional case is proved in   \cite{Stepanov07}.
For quotient singularities in arbitrary dimensions
Theorem \ref{c-dlt} is proved in \cite{KS11}

It is  interesting to connect 
algebraic properties of singularities with topological properties of the
dual complex.
Using \cite{steenbrink} one can show that if
$(0\in X)$ is a rational singularity then
$\DR(0\in X)$ is $\mathbb{Q}$-acyclic, that is,  
$H^i\bigl(\DR(0\in X), \mathbb{Q}\bigr)=0$ for $i>0$. 
If $(0\in S)$ is a rational surface singularity
then $\DR(0\in S)$ is contractible. 
\cite{Stepanov06} asked if $\DR(0\in X)$ is contractible
for higher dimensional isolated rational singularities as well,
however the opposite turned out to be true.
 A finite simplicial complex $C$ occurs as a 
dual complex of a rational singularity iff 
$C$ is $\mathbb{Q}$-acyclic; see 
\cite[Thm.42]{KK11}, 
 \cite[Thm.8]{Kollar12} and  \cite[Thm.3]{Kollar13}.

By \cite[Thm.2]{Kollar13}
every finite simplicial complex occurs as a  dual complex of some 
isolated singularity.

After this, the dlt case seemed a natural candidate. It was explicitly
asked in \cite[Question~69]{Kollar12} and in
\cite{dF12}.
Our results also answer \cite[Question~68]{Kollar12}
and solve \cite[Conjecture~70]{Kollar12}.

 While we prove 
contractibility of various dual complexes,
the contraction and even the end result depend on auxiliary choices.
Nonetheless, in many examples there seems to be a canonical choice
and it would be useful to understand this situation better.

We were led to conjecturing Theorems \ref{c-dlt}--\ref{t-RC} partly through
arithmetic considerations. If $(0\in X)$ and the resolution $Y\to X$ are
 defined over a field $k$ then the Galois group
 $\operatorname{Gal}(\bar k/k)$ acts on  the dual complex $\DR(0\in X_{\bar k})$
over an algebraic closure $\bar k$.
The key results of the papers \cite{Kollar07,HX09,LX11} say that
this action has a fixed point. This suggests that in these settings
the dual complex has a natural retraction to a particular point
that is Galois-invariant. See Theorems \ref{c-dlt-gal}--\ref{t-RC-gal}
for the corresponding generalizations.

Our methods rely on the minimal model program or MMP. In general, 
the property of a pair to have simple normal crossings is not preserved 
 during MMP but being dlt is preserved.
Our main technical result studies how the
dual complex $\D(\Delta^{=1})$ changes as we run MMP
on a dlt pair  $(X,\Delta)$. 
In retrospect, the connectedness theorems 
\cite[Thm.5.48]{KM98} and \cite[Sec.4.4]{kk-singbook}
appear as the simplest special cases of Theorem \ref{t-sdr}.

\begin{ack}
We thank O.~Fujino  for helpful  comments and references.
Partial financial support  to TdF was provided  by  the NSF under grant number
 DMS-0847059 and by the Simons Foundation
Partial financial support  to JK  was provided  by  the NSF under grant number 
DMS-07-58275 and by the Simons Foundation. 
Partial financial support  to CX was provided  by  the NSF under grant number
DMS-1159175, by the 
Warnock Presidential Endowed Chair and by the Chinese government grant
``Recruitment Program of Global Experts.''
\end{ack}

\begin{say}[Notation and Conventions]\label{nor.conv.say}
We follow the notation and definitions of the books \cite{KM98, kk-singbook}.
We work over an algebraically closed field of characteristic 0. 

A \emph{pair} $(X,\Delta)$ consists of a normal variety $X$ and a
 $\mathbb Q$-divisor $\Delta$ on it. 
If all coefficients are in $[0,1]$ (resp.\ $(-\infty, 1]$), 
then we say that $\Delta$ is a \emph{boundary} (resp.\ \emph{sub-boundary}).

Let $f\colon Y\to X$ be a birational morphism.
If  $K_X+\Delta$ is $\mathbb{Q}$-Cartier then the formulas
$$
K_{Y}+\Delta_Y\sim_{\mathbb Q} f^*(K_X+\Delta)
\quad\mbox{and}\quad f_*\Delta_Y=\Delta
$$
define the $\mathbb Q$-divisor $\Delta_Y$, called the \emph{log pull-back}
of $\Delta$. (See  \cite[Sec.2.3]{KM98} or \cite[2.6]{kk-singbook} for details.)
For a divisor $E\subset Y$ its {\it discrepancy, } denoted by $a(E, X, \Delta)$,
is  the negative of its coefficient in $\Delta_Y $.

A pair $(X,\Delta)$ is {\it log canonical,} abbreviated as {\it lc,}
if $a(E, X, \Delta)\geq -1$ for every divisor $E$ and every
birational morphism $f\colon Y\to X$.

Let $(X,\Delta)$ be an lc pair. An irreducible subvariety
$Z\subset X$ is an  \emph{lc center} 
iff there is a  birational morphism $f\colon Y\to X$ 
and a divisor $E\subset Y$ such that $a(E,X,\Delta)=-1$ and
$f(E)=Z$.

We say that $(X,\Delta)$ is a \emph{simple normal crossing pair} if $X$ is 
smooth and $\Delta$ has simple normal crossing support.
Given any pair $(X,\Delta)$, there is a largest open set
$X^{\rm snc}\subset X$, called the 
{\it simple normal crossing locus} or {\it snc locus}
such that $\bigl(X^{\rm snc},\Delta|_{X^{\rm snc}}\bigr)$ is a 
simple normal crossing pair.

A log canonical pair $(X,\Delta)$ is  {\it divisorial log terminal,} 
abbreviated as  {\it dlt,} if none of the lc centers of
$(X,\Delta)$ is contained in $X\setminus X^{\rm snc}$.

Two  pairs $(X_i,\Delta_i)$ are called \emph{crepant birational equivalent} if 
there are proper, birational morphisms $f_i:Y\to X_i$ such that
 the log pull back of $\Delta_1$ on $Y$ 
equals the log pull back of $\Delta_2$ on $Y$.
\end{say}

\section{The dual complex of a dlt pair}

\begin{defn}\label{complexes.defn}
For cell complexes, we follow the terminology of
\cite{Hatcher}, see especially Section  2.1 and the Appendix.

The notion of an (unordered)  {\it $\Delta$-complex} is defined inductively.
A $0$-dimensional   $\Delta$-complex is a collection of 
points called {\it vertices.}
If  $k$-dimensional   $\Delta$-complexes  
and their attaching maps are already defined,
we obtain 
$(k+1)$-dimensional  $\Delta$-complexes as follows.

Let $C_k$ be a $k$-dimensional  $\Delta$-complex
and $\{S_i:i\in I_{k+1}\}$  a collection of
$(k+1)$-dimensional simplices. The boundary $\partial S_i$
is a $k$-dimensional   cell complex. An attaching or characteristic map
is a map $\tau_i:\partial S_i\to C_k$. 
Identifying the points of $\partial S_i$ with their image
in $C_k$ for every $i$ gives a $(k+1)$-dimensional  $\Delta$-complex
$C_{k+1}$. 
The $\leq k$-dimensional cells of $C_{k+1}$ are the cells in $C_k$ and
the images of the $S_i$ give the   $(k+1)$-dimensional cells.
For $j\leq k$ the cell complex $C_j$ is called the {\it $j$-skeleton}
of $C_{k+1}$.

If all the attaching maps are embeddings, then 
the resulting object is a {\it regular cell complex.}

A regular cell complex is called a {\it simplicial complex}
if the intersection of any 2 simplices  $S_1, S_2$ is a  single
(possibly empty) simplex $S_3$
which is a facet in both of them.
In Figure~\ref{fig:complexes} below we have 3  $\Delta$-complexes of dimension 1.
The first is a $\Delta$-complex that is not 
regular, the second is  regular but
 not simplicial and the third is  simplicial.

\begin{figure}[htb]
\begin{center}
\begin{tikzpicture}

\draw (1,1) circle (0.8);
\fill (0.2,1) circle (2pt) node[left] {$v_1$};

\draw (5,1) circle (0.8);
\fill (4.2,1) circle (2pt) node[left] {$v_1$};
\fill (5.8,1) circle (2pt) node[right] {$v_2$};

\draw (9,1) circle (0.8);
\fill (8.2,1) circle (2pt) node[left] {$v_1$};
\fill (9.8,1) circle (2pt) node[right] {$v_2$};
\fill (9,1.8) circle (2pt) node[above] {$v_3$};
\fill (9,0.2) circle (2pt) node[below] {$v_4$};

\end{tikzpicture}
\end{center}
\caption{}
\label{fig:complexes}
\end{figure}

 Let $D$ be a regular cell complex. For a simplex
$v\subset  D$ let $\st(v)$ be the (open) {\it star} of $v$,
that is, the union of the interiors of all cells whose closure contains $v$.
Its closure, denoted by  $\cst(v)$, is the {\it closed star}. 

If $D$ is a simplicial complex then their difference
$\lk(v)=\cst(v)\setminus \st(v)$ is the {\it link} of $v$.
Note that the stellar subdivision of the closed star $\cst(v)$ is a cone over 
the link $\lk(v)$. (This  fails if $D$ is not simplicial.)
\end{defn}

\begin{defn}\label{dual.cpx.defn}
Let $Z=\bigcup_{i\in I}Z_i$ be a pure dimensional scheme with
irreducible components $Z_i$. Assume that
\begin{enumerate}
\item each $Z_i$ is normal  and
\item for every $J\subset I$, if $\cap_{i\in J}Z_i$ is nonempty, then every connected component
of  $\cap_{i\in J}Z_i$ is  irreducible and has codimension $|J|-1$ in $Z$. 
\end{enumerate}
Note that  assumption (2) implies the following.
\begin{enumerate}\setcounter{enumi}{2}
\item For every $j\in J$, 
every irreducible component
of $\cap_{i\in J}Z_i$ is contained in a 
unique irreducible component
of $\cap_{i\in J\setminus\{j\}}Z_i$.
\end{enumerate}
The {\it dual complex} $\D(Z)$ of $Z$ is the regular cell complex obtained
as follows.
The vertices are the irreducible components of $Z$ and
to each irreducible component
of $W\subset \cap_{i\in J}Z_i$ we associate a  cell of dimension $|J|-1$.
This cell is usually denoted by $v_W$.

The attaching map is given by condition (3).
Note that $\D(Z)$ is a  simplicial complex
iff $\cap_{i\in J}Z_i$ is irreducible  (or empty)
for every $J\subset I$.

Let $X$ be a variety and $E$ a divisor on $X$.
If $\operatorname{Supp}(E)$ satisfies the conditions
(1--2), then 
$\D(E):=\D\bigl(\operatorname{Supp}(E)\bigr)$
is called the {\it dual complex}  of $E$.

Note that  conditions (1--2) are satisfied in three important cases:
\begin{enumerate}\setcounter{enumi}{3}
\item  $X$ is a smooth and  $E$ is a simple normal crossing divisor.
\item  $(X,\Delta)$ is a dlt pair and 
$E:=\Delta^{=1}$  is the set of divisors whose coefficient in
$\Delta $ equals 1.
\item  We introduce quotient-dlt pairs in  Definition \ref{qdlt.defn}
with conditions (1--2) in  mind.
\end{enumerate}
Here the claim (4) is clear and (5) is proved in
\cite[Sec.3.9]{Fujino}; see also
\cite[Thm.4.16]{kk-singbook}. 
In the dlt case, let
$X^{\rm snc}\subset X$ be the simple normal crossing locus.
Then $\D\bigl(\Delta^{=1}\bigr)=\D\bigl((\Delta|_{X^{\rm snc}})^{=1}\bigr)$,
thus the dual complex is insensitive to the singularities of
$(X,\Delta)$.

In the above cases (4--6) the cells of $\D(E)=\D(\Delta^{=1})$
are identified with the log canonical centers of
$(X,\Delta)$. Frequently these are also called the
{\it strata} of $E$ or of $\Delta^{=1}$.

Even if $X$ is  smooth and  $E=\bigcup_{i\in I}E_i$ is a 
simple normal crossing divisor, the dual complex 
$\D(E)$ need not be  a simplicial complex, but this can be achieved
after some blow-ups as in Remark \ref{make.snc.say}.

It is possible to define the dual complex even if $Z$ does not satisfy
the conditions (1--2). However, there seem to be several ways to
do it and we do not know which variant is the most useful.
In this paper we only use the cases arising from (4--6). 
\end{defn}

\begin{say}[Blowing-up and the dual complex]\label{bu.and.baryc.say}
Let $X$ be a  smooth variety and  $E=\bigcup_{i\in I}E_i$  a 
simple normal crossing divisor. Let $Z\subset X$ be a smooth,
irreducible subvariety that has only simple normal crossing  with $E$;
see \cite[3.24]{k-res}.

Let $\pi:B_ZX\to X$ denote the blow up of $Z$ with exceptional divisor
$E'_0$  (assuming $0\not\in I$).
 Let $E'_i:=\pi^{-1}_*(E_i)$ denote the birational transform of
$E_i$ and $E':=\pi^{-1}_*(E)= \bigcup_{i\in I}E'_i$.
Then $E'_0\cup E'$ is a 
simple normal crossing divisor.
By a direct computation we see the following.
\begin{enumerate}
\item If  $Z$ is  a stratum of $E$ then
$\D\bigl(E'_0\cup E'\bigr)$ is obtained from $\D\bigl(E\bigr)$
by the {\it stellar subdivision} of $\D\bigl(E\bigr)$
corresponding to an interior point of  the cell  $v_Z$.
\item If $Z$ is not a stratum of $E$ then
$\D\bigl(E'\bigr)=\D\bigl(E\bigr)$.
\item If $Z$ is not a stratum of $E$ but $Z\subset E$ then
$\D\bigl(E'_0\cup E'\bigr)$ is obtained from $\D\bigl(E\bigr)$
as follows. 

Let $E_Z$ be the smallest stratum that contains $Z$
and $v_Z$  the corresponding simplex. Note that
$E_Z$ is an irreducible component of some intersection
$\bigcap_{i\in J} E_i$.
Let  $\D(Z)$ denote the dual  complex of
$\sum_{i\in I\setminus J}E_i|_Z$. 
Then the dual complex $\D\bigl(E'|_{E'_0}\bigr)$ is
the join  $v_Z* \D(Z)$ (see e.g. \cite[Page 9]{Hatcher}).
There is a natural map
$\tau_L:\D(Z)\to \Lk(v_Z)$ hence we get a map
$\tau_S: v_Z* \D(Z)\to \cst(v_Z)$.
We can identify $v_Z* \D(Z)$
with a subcomplex of the  cone
$\cone\bigl(v_Z* \D(Z)\bigr) $ and attach the latter
 to $\D\bigl(E\bigr)$
using $\tau_S$ to get $\D\bigl(E'_0\cup E'\bigr)$.

 Since the cone over the join  retracts (even collapses) to the
join, we see that
$\D\bigl(E'_0\cup E'\bigr)$  retracts (even collapses) to $\D\bigl(E\bigr)$.
(However, they need not be PL homeomorphic;
even their dimension can be different.)
\end{enumerate}

{\it Note on terminology.}  Let $D$ be a regular 
cell complex and $p\in D$ a point.
The {\it stellar subdivision} of $D$ with center $p$ is obtained as follows.
\begin{enumerate}
\item[i)] The closed cells not containing $p$ are unchanged.
\item[ii)] If  $v$ is a closed cell containing $p$, we replace it
by all the cells that are spans $\langle p, w\rangle$
where $w\subset v$ is any face non containing $p$.
\end{enumerate}
(Some authors seem to call this a barycentric subdivision, but
this goes against standard usage in PL topology  \cite{Hatcher, Spanier}.)

\medskip

Let $(X, \Delta)$ be an snc pair where $\Delta$ is a sub-boundary
and apply the above observations to $E:=\Delta^{=1}$.
Write  
$$
K_{B_ZX}+\Delta_Z\sim_{\mathbb Q}\pi^*\bigl(K_X+\Delta\bigr)
$$
and set $E_Z:=(\Delta_Z)^{=1}$.
Note that $E_Z=E'$ if $Z$ is not a stratum of $E$ and
$E_Z=E'_0\cup E'$ if $Z$ is a stratum of $E$. Thus we conclude the following.
\begin{enumerate}\setcounter{enumi}{3}
\item  If $\bigl(K_{B_ZX},\Delta_Z\bigr)$ is obtained from
$(K_X,\Delta)$ by blowing up  a subvariety $Z\subset X$ that has
 simple normal crossing  with $\Delta$ then
$\D\bigl(\Delta_Z^{=1}\bigr)$ is PL homeomorphic to
$\D\bigl(\Delta^{=1}\bigr)$.
\end{enumerate}
\end{say}

\begin{rem} \label{make.snc.say}
The barycentric subdivision of any dual complex is simplicial. 
For  the dual complex of a simple normal crossing divisor,
it can be realized by blow-ups as follows.

Let $X$ be a  smooth variety and  $E=\bigcup_{i\in I}E_i$  a 
simple normal crossing divisor.
Let $Z_i\subset X$ denote the union of all $i$-dimensional strata
of $E$. Consider a sequence of blow-ups
$$
\Pi: \tilde X:= X_{n-1}\stackrel{\pi_{n-2}}{\longrightarrow}
X_{n-2}\to \cdots \to X_1\stackrel{\pi_{0}}{\longrightarrow} X_0:=X
$$
where $\pi_i:X_{i+1}\to X_i$ denotes the blow-up of the
birational transform 
$ (\pi_0\circ\cdots \pi_{i-1})^{-1}_*Z_i\subset X_i$.

Each blow-up center is smooth and $\Pi^{-1}E$ is
a simple normal crossing divisor whose 
dual complex 
$\D\bigl( \Pi^{-1}E\bigr)$ is  the stellar subdivision of $\D(E)$. 

Finally, we know that a barycentric subdivision can be written as a sequence of stellar subdivisions. 
\end{rem}

The following
invariance result for the dual complexes of
crepant-birational dlt pairs is a considerable
strengthening of \cite[4.35]{kk-singbook}.

\begin{prop} \label{crep.bir.dlt.models}
  Let $f_i: (X_i, \Delta_i)\to S$ be proper morphisms.
Assume that the $(X_i, \Delta_i)$ are dlt and crepant-birational
to each other over $S$. Then the dual complexes
$\D\bigl(\Delta_i^{=1}\bigr)$ are PL homeomorphic to
each other.
\end{prop}

Proof. By \cite{szabo}, every dlt pair $(X,\Delta_X)$
has a log resolution  $g: (Y, \Delta_Y)\to (X,\Delta_X)$
such that  $\D\bigl(\Delta_Y^{=1}\bigr)$ is naturally identified
with $\D\bigl(\Delta_X^{=1}\bigr)$. Thus we may assume that
the $(X_i, \Delta_i)$ are snc pairs but now the
$\Delta_i $ are only sub-boundaries.

Next we use the  factorization theorem of \cite{AKMW02}
which says that there is a sequence of 
smooth blow-ups  as in
Paragraph \ref{bu.and.baryc.say} and their inverses
$$
X_1=Y_0\stackrel{\pi_0}{\map} Y_1\stackrel{\pi_1}{\map} \cdots 
\stackrel{\pi_{m-1}}{\map} Y_m  \stackrel{\pi_m}{\map} Y_{m+1}
\stackrel{\pi_{m+1}}{\map} \cdots \stackrel{\pi_{r-1}}{\map} Y_r=X_2.
$$
Moreover, we may assume that the induced maps 
$\pi_0^{-1}\circ \cdots\circ \pi_{i-1}^{-1}: Y_i\map X_1$
 are morphisms for $i\leq m$ and 
the induced maps 
$\pi_{r-1}\circ \cdots\circ \pi_{i}:Y_i\map X_2$ are morphisms for $i\geq m$.

Let $\Theta_i$ be the log pull-back of $\Delta_1$
for  $i\leq m$ and the log pull-back of $\Delta_2$
for  $i\geq m$ by the above morphisms.
Note that the two definitions of $\Theta_m$ agree since
the $(X_i, \Delta_i)$ are  crepant-birational
to each other over $S$.

We use (\ref{bu.and.baryc.say}.4) to show that at each step
$\D\bigl(\Theta_i^{=1}\bigr)$ changes either by a stellar subdivision or
its inverse.

Thus the $\D\bigl(\Delta_i^{=1}\bigr)$ are obtained from each other
by repeatedly adding and removing stellar subdivisions.\qed

\begin{comp} \label{crep.bir.dlt.models.rem}
Using the above notation, let $Z\subset S$ be a closed subset.
Let $\Theta_{i,Z}^{=1}\subset \Theta_i^{=1}$ be the union of those
divisors that are contained in $f_i^{-1}(Z)$.  Assume further that
every lc center of $\Theta_i^{=1}$ contained in $f_i^{-1}(Z)$
is also an lc center of $\Theta_{i,Z}^{=1}$. 

Then the above proof shows that the dual complexes
$\D\bigl(\Delta_{i,Z}^{=1}\bigr)$ are PL homeomorphic to
each other.
\end{comp}

\medskip

Although we do not need this, it is worth remarking that
if $X_1\map X_2$ is an isomorphism in codimension 1 then
we can go between the $X_i$ by a series of {\it flops}.
The topological analogs of these are the
{\it Pachner moves} or {\it bistellar flips,}
see \cite{pachner}.
\medskip

We will need several types of partial resolutions of a pair 
$(X,\Delta) $.

\begin{defn}\label{dlt-mod.defn}
Let $X$ be a normal variety and $\Delta$ a boundary on $X$.
Let $g':(X',\Delta')\to (X,\Delta)$ be a proper, birational morphism
where $\Delta'=E+(g')^{-1}_*\Delta$ and 
$E$ is the sum of all divisors in $\ex(g')$.
We say that $g':(X',\Delta')\to (X,\Delta)$ is a 
$$
\left.%
\begin{array}{c}
\mbox{dlt} \\
\mbox{qdlt} \\
\mbox{lc} 
\end{array}
\right\}\mbox{ modification if $(X',\Delta')$ is }\left\{
\begin{array}{c}
\mbox{dlt} \\
\mbox{qdlt} \\
\mbox{lc} 
\end{array}
\right\}\mbox{ and $K_{X'}+\Delta'$ is }\left\{%
\!
\begin{array}{l}
\mbox{$f$-nef.}\\
\mbox{$f$-nef.}\\
\mbox{$f$-ample.} 
\end{array}
\right.
$$
We frequently denote a dlt modification by
$g^{\rm dlt}: \bigl(X^{\rm dlt}, 
\Delta^{\rm dlt}\bigr)\to (X, \Delta)$
and an  lc modification by
$g^{\rm lc}: \bigl(X^{\rm lc}, 
\Delta^{\rm lc}\bigr)\to (X, \Delta)$.
For  qdlt, see Definition \ref{qdlt.defn}.

Every pair $(X, \Delta) $ 
such that $K_X+\Delta$ is $\mathbb Q$-Cartier 
has a unique log canonical modification and
(usually non-unique) dlt modifications by \cite{OX12}.
(Conjecturally, both  should exist for
any $(X, \Delta)$.) We also remark that because we require that $K_X+\Delta$ is nef, this is indeed stronger than the existence of dlt blow up proved by Hacon (see \cite[10.4]{Fuj}).

 Note that $\ex(g^{\rm lc})$ has pure codimension 1,
but this need not hold for $\ex(g^{\rm dlt})$.
However, we prove in Lemma \ref{3.3.1.4.HMX12.lem}
 that there is a dlt modification
such that $\ex(g^{\rm dlt})$ has pure codimension 1.
\end{defn}

We can now define the most important dual complexes
associated to a pair $(X, \Delta) $.

\begin{defn}\label{DR.W.X.def}
Let $X$ be a normal variety and $W\subset X$ a closed subvariety.
Let $p:Y\to X$ be a resolution of singularities such that
$\supp p^{-1}(W)$ is a simple normal crossing divisor.
By \cites{Stepanov06,Thuillier07,Stepanov07, ABW11, Payne11},
the simple-homotopy class of the dual complex
$\D\bigl(\supp p^{-1}(W)\bigr)$ is independent of
the choice of $p$. We denote it by
$\DR(W\subset X)$.
\end{defn}

\begin{defn}\label{DMR.defn}
The {\it non-klt locus} of $(X, \Delta) $ is the unique smallest subscheme
$W\subset X$ such that $\bigl(X\setminus W, \Delta|_{X\setminus W}\bigr)$ is klt.
It is frequently denoted by  $\operatorname{non-klt} (X, \Delta) $.
It can be written as 
$g^{\rm dlt}\bigl((\Delta^{\rm dlt})^{=1}\bigr)$ for any
dlt modification.

Any two dlt modifications are crepant-birational
to each other over $X$. (Although not explicitly stated,
this is what the proof of \cite[3.52]{KM98} gives.)
 Hence, by Proposition \ref{crep.bir.dlt.models},
 the dual complex
$$
\DMR(X,\Delta):=\D\bigl((\Delta^{\rm dlt})^{=1}\bigr)
$$
is independent of the choice of $g^{\rm dlt}:X^{\rm dlt}\to X$, up-to
PL homeomorphism.

We stress that $\DMR(X,\Delta) $ is a PL homeomorphism equivalence class
but $\DR(W\subset X)$ is a simple-homotopy equivalence class.

Note that $\DMR(X,\Delta) $ does depend on
$\Delta$, not just on the pair $(W\subset X)$. 
This will be quite useful to us. In some cases we will be able to compute
$\DR(W\subset X)$ by choosing a suitable $\Delta$
and then computing $\DMR(X,\Delta) $.
\end{defn}

\section{Running MMP}

Let $(X,\Delta)$ be an lc pair and
 $f \colon X \dasharrow Y$  a step  (a divisorial contraction or a flip) 
in an $(X,\Delta)$-MMP. 
Set $\Delta_Y := f_*\Delta$; then $(Y, \Delta_Y)$ is also an lc pair.
Furthermore, if $(X,\Delta)$ is dlt  then so is
$(Y, \Delta_Y)$ by \cite[3.44]{KM98}.

Once we introduce quotient-dlt pairs in Section \ref{sec.quot.dlt},
we see that all the proofs and results in this section hold for
quotient-dlt pairs.

Let $Z\subset X$ be an lc center of $(X,\Delta)$.
We say that $f$ {\it contracts} $Z$ if $Z\subset \ex(f)$.
If $Z\not\subset \ex(f)$ then
$f_*(Z)$ is an lc center of $(Y, \Delta_Y)$.
Since discrepancies strictly increase for every divisor whose
center is contained in $\ex(f)$, we see that
every lc center of $(Y, \Delta_Y)$ is obtained this way from a
non-contracted lc center of $(X,\Delta)$.
Using this observation repeatedly, we conclude the following.

\begin{lem}\label{subcomp.lem} Let $(X,\Delta)$ be a dlt  pair and
 $f \colon X \dasharrow Y$  a birational map obtained
by running an $(X,\Delta)$-MMP. 
Set $\Delta_Y := f_*\Delta$. 
Then the dual complex $\D(\Delta_Y^{=1})$ is 
naturally a subcomplex of $\D(\Delta^{=1})$. \qed
\end{lem}

In general, the inclusion $\D(\Delta_Y^{=1})\subset \D(\Delta^{=1})$
is not a homotopy equivalence, but this happens in
many interesting cases. The following examples illustrate the main
possibilities.

\begin{exmp} Let $Z=(x_1x_2-x_3x_4=0)\subset {\mathbb C}^4$ be the quadric cone.
Consider the planes 
$A_1:=(x_1=x_3=0), A_2:=(x_2=x_4=0), B_1:=(x_1=x_4=0), B_2:=(x_2=x_3)=0$.
The two small resolutions are
$Y':=B_{A_1}X=B_{A_2}X$ and $Y'':=B_{B_1}X=B_{B_2}X$.
By explicit computation we see the following.

(1) $f:Y'\map Y''$ is a flip for 
$\bigl(Y', A'_1+A'_2\bigr)$. The corresponding dual complexes are as in Figure~\ref{fig1}.

\begin{figure}[htb]
\begin{center}
\begin{tikzpicture}

\path[draw] (2,0)--(0,0);
\fill (0,0) circle (2pt) node[below] {$v_{A_1'}$};
\fill (2,0) circle (2pt) node[below] {$v_{A_2'}$};

\draw[dotted,->] (4,0) -- (6,0);

\fill (8,0) circle (2pt) node[below] {$v_{A_1''}$};
\fill (10,0) circle (2pt) node[below] {$v_{A_2''}$};

\end{tikzpicture}
\end{center}
\caption{}
\label{fig1}
\end{figure}

This is not a homotopy equivalence.

(2) $f:Y'\map Y''$ is a flip for 
$\bigl(Y', A'_1+A'_2+B'_1\bigr)$. The corresponding dual complexes are as in Figure~\ref{fig2}.

\begin{figure}[htb]
\begin{center}
\begin{tikzpicture}

\path[draw, fill=gray!20] (2,0)--(0,0)--(0,2)--cycle;
\fill (0,0) circle (2pt) node[below] {$v_{B_1'}$};
\fill (0,2) circle (2pt) node[above] {$v_{A_2'}$};
\fill (2,0) circle (2pt) node[below] {$v_{A_1'}$};

\draw[dotted,->] (4,1) -- (6,1);

\path[draw] (10,0)--(8,0)--(8,2);
\fill (8,0) circle (2pt) node[below] {$v_{B_1''}$};
\fill (8,2) circle (2pt) node[above] {$v_{A_2''}$};
\fill (10,0) circle (2pt) node[below] {$v_{A_1''}$};

\end{tikzpicture}
\end{center}
\caption{}
\label{fig2}
\end{figure}

This is a homotopy equivalence, even a collapse.

(3) $f:Y'\map Y''$ is a flop for 
$\bigl(Y', A'_1+A'_2+B'_1+B'_2\bigr)$. The corresponding dual complexes are as in 
Figure~\ref{fig3}.

\begin{figure}[htb]
\begin{center}
\begin{tikzpicture}

\draw[fill=gray!20] (0,0) rectangle (2,2);
\fill (0,0) circle (2pt) node[below] {$v_{B_1'}$};
\fill (0,2) circle (2pt) node[above] {$v_{A_2'}$};
\fill (2,0) circle (2pt) node[below] {$v_{A_1'}$};
\fill (2,2) circle (2pt) node[above] {$v_{B_2'}$};
\draw (0,2) -- (2,0);

\draw[dotted,->] (4,1) -- (6,1);

\draw[fill=gray!20] (8,0) rectangle (10,2);
\fill (8,0) circle (2pt) node[below] {$v_{B_1''}$};
\fill (8,2) circle (2pt) node[above] {$v_{A_2''}$};
\fill (10,0) circle (2pt) node[below] {$v_{A_1''}$};
\fill (10,2) circle (2pt) node[above] {$v_{B_2''}$};
\draw (8,0) -- (10,2);

\end{tikzpicture}
\end{center}
\caption{}
\label{fig3}
\end{figure}

This is a PL homeomorphism, a composite of a stellar subdivison 
(whose center is the center of the square) and
its inverse.
\end{exmp}

In order to describe the general case precisely,
we need some concepts from simplicial complex theory,
especially the notion of  collapse.

\begin{defn} \label{collapse.defn}
Let $D$ be a regular cell complex.
Let $v$ be a cell in $D$ and $w$ a face of $v$. We say that
$(v,w)$ is a {\it free pair} if $w$ is not the face of any other
cell in $D$. The {\it  elementary collapse} of $(D, v,w)$ 
is the regular complex obtained from $D$ by removing
the interiors of the cells $v$ and $w$. This is clearly a homotopy equivalence.
A sequence of elementary collapses is called a 
{\it  collapse.}
A  regular complex $D$ is {\it  collapsible} if it
collapses to a point.

A map $g:C_1\to C_2$ of regular complexes is a
{\it simple-homotopy equivalence} if 
it is homotopic to a map 
obtained  by a sequence of  
elementary collapses and their inverses.
If the $C_i$ are simply connected, then 
every homotopy equivalence is a simple-homotopy equivalence.
However, in general not every homotopy equivalence is a 
simple-homotopy equivalence; the difference
is measured by the {\it Whitehead torsion.} 
For details and for proofs  see \cite{Coh73}.
\end{defn}

The following is our key result relating extremal contractions
to collapses of the dual complex.

\begin{thm} \label{subcomp.collapse.thm}
Let $(X,\Delta)$ be dlt  
and  $f \colon X \dasharrow Y$   a 
  divisorial contraction or  flip
corresponding to a $(K_X+\Delta)$-negative extremal ray $R$.
 Set $\Delta_Y := f_*\Delta$. 
Assume that  there is a prime divisor $D_0\subset \Delta^{=1}$ such that  $(D_0\cdot R)>0$.

Then $\D(\Delta^{=1})$ collapses to $\D(\Delta_Y^{=1})$. 
\end{thm}

Proof.  We distinguish two types of  contracted lc centers.

{\it Case 1.} Let $Z$ be a contracted lc center such that
$Z\subset D_0$.
 Then there is a subset $0\not\in J\subset I$
such that $Z$ is an irreducible component of $D_0\cap \bigcap_{i\in J}D_i$. Thus
there is a unique irreducible component $Z^+$ of $\cap_{i\in J}D_i$ such that
$Z$ is an irreducible component of $ D_0\cap Z^+$.   We claim that $Z^+\subset \ex(f)$.

Indeed, if $Z^+\not\subset \ex(f)$ then
$f_*(Z^+)\subset Y$ is an lc center which has nonempty intersection
with $\ex(f^{-1})$. Since $(D_0\cdot R)>0$, we see that
$f_*(D_0)$ contains $\ex(f^{-1})$. 
Denote by $h:X\to X_1$ the divisorial or flip contraction morphism. Recall that $Z$ is also a connected component of  $ D_0\cap Z^+$. Therefore, Proposition \ref{p-connect} implies that if there is another component of $D_0\cap Z^+$ whose image intersects  $h(Z)$, then itself must intersect with $Z$ which is absurd. Hence we conclude that $Z$ is the same as $D_0\cap Z^+$  over a neighborhood of $h(Z)$. 
Thus over the neighborhood of $h(Z)$,
$f_*(Z^+)\cap f_*(D_0)$ is  a nonempty subset of
 $\ex(f^{-1})$. 
By \cite{Ambro03}, $f_*(Z^+)\cap f_*(D_0)$ is  a union of lc centers.
However, $\ex(f^{-1})$ does not contain any lc centers,
a contradiction.

{\it Case 2.} Let $W$ be a contracted lc center such that $W\not\subset D_0$. 
We claim that 
$W^-:=W\cap D_0$ is a (nonempty, irreducible) lc center.
Thus $W^-$ is also contracted and the operations
$Z\mapsto Z^+$ and $W\mapsto W^-$ are inverses.

To see this, let $h:X\to X_1$ denote the contraction.
First apply Proposition \ref{p-connect} to the generic point of the Stein factorization of the morphism $(W,W\cap D_0)\to h(W)$.
It shows that $W\cap D_0$ has a unique irreducible component $W^-$
that dominates $h(W)$. If $W\cap D_0$ has any other irreducible component
$W^*$, this would give a contradiction using Proposition \ref{p-connect} over
the generic point of $h(W^*)$.

Let  $M\subset \D(\Delta^{=1})$ be the union of the interiors of all
cells $v_Z$ where $Z\subset \ex(f)$ is a contracted  lc center. 
 Thus 
$$
\D(\Delta_Y^{=1})=\D(\Delta^{=1})\setminus M.
$$
We show that all the cells in $M$ can be collapsed,
starting with the largest dimensional ones.

By the above considerations, the cells in $M$
naturally come in pairs  
$$
\bigl(\langle v_{D_0}, v_W\rangle, v_W\bigr)
$$
where $W\subset X$ is a contracted lc center
not contained in $D_0$. The other cell $\langle v_{D_0}, v_W\rangle$,
 spanned by $v_{D_0}$ and $v_W$, is the same as  $v_{D_0\cap W}$. 
We call $v_W$ {\it link-type} (since $v_W\subset \lk(v_{D_0})$)
and $\langle v_{D_0}, v_W\rangle$
{\it star type} (since $\langle v_{D_0}, v_W\rangle\subset \cst(v_{D_0})$).

Note further that if an lc center $W$ is contracted, then
so is every  lc center contained in $W$.
Dually, if $v_W\in M$ is a facet of a cell $v_V$ then also
$v_V\subset M$.

Let  $\bigl(\langle v_{D_0}, v_W\rangle, v_W\bigr)$ be a maximal
dimensional pair in $M$. If $v_W$ is a face of a cell
$v_V$ then $v_V$ is also in $M$.
By maximality of dimension, $v_V$ is of star-type,
thus  $v_V=\langle v_{D_0}, v_W\rangle $.
Thus $\bigl(\langle v_{D_0}, v_W\rangle, v_W\bigr)$
is a free pair and it can be collapsed. 

Set $m:=\dim \D(\Delta_Y^{=1})$. Note that these collapses
remove all star-type cells of dimension $m$ and all
link-type cells of dimension $m-1$.

Next we look at an $(m-1)$-dimensional pair
$\bigl(\langle v_{D_0}, v_W\rangle, v_W\bigr)$.
Assume that  $v_W$ is a face of a cell
$v_V$. It can not be of link-type since we have already removed
all link-type cells of dimension $m-1$. Thus 
$v_V$ is of star-type and so 
$v_V=\langle v_{D_0}, v_W\rangle $.
Hence $\bigl(\langle v_{D_0}, v_W\rangle, v_W\bigr)$ is again a free
pair and can be collapsed.

Iterating this 
completes the proof.\qed

\begin{rem} It is interesting to see what happens in Theorem 
\ref{subcomp.collapse.thm}
if $(D_i\cdot R)<0$ for every $D_i\subset \Delta^{=1}$.

Let $Z$ be a contracted lc center that is an irreducible component of 
$\bigcap_{i\in J}D_i$. Then $\ex(f)\subset D_i$ for every $i$,
thus $\ex(f)=Z$. Hence $\D(\Delta_Y^{=1})$
is obtained from $\D(\Delta^{=1})$ by removing the cell $v_Z$.
This is not a homotopy equivalence, even the Euler characteristic changes
by 1.

The  general case when $(D_i\cdot R)\leq 0$ 
for every $D_i\subset \Delta^{=1}$ is more complicated.
\end{rem}

\medskip

In order to use Theorem \ref{subcomp.collapse.thm}
we need to find conditions that ensure the
existence of such a divisor $D_0$ at each step of an MMP.

\begin{lem} \label{numtrov.div.lem}
 Let $(X,\Delta)$ be dlt and  $g:X\to S$ a morphism.
Assume that 
 there is a
numerically $g$-trivial effective divisor $A$
whose support equals $\Delta^{=1}$.
Let  $f \colon X \dasharrow Y$  be a 
  divisorial contraction or  flip
corresponding to a $K_X+\Delta$-negative extremal ray $R$ over $S$.
Then 
\begin{enumerate}
\item either $\ex(f)$ does not contain any lc centers
\item or there is a prime divisor $D_0\subset \Delta^{=1}$ such that  $(D_0\cdot R)>0$.
\end{enumerate}
\end{lem}

Proof. Let $Z$ be a contracted lc center
that is an irreducible component of $\cap_{i\in J}D_i$.
We are done if $(D_i\cdot R)>0$ or some $i$.

Otherwise $(D_i\cdot R)\leq 0$ for every $D_i\subset \Delta^{=1}$
and so  $(A\cdot R)\leq 0$ with equality holding only if 
 $(D_i\cdot R)= 0$ for every $D_i$.
In particular $(D_i\cdot R)=0$ for every $i\in J$.

If $f$  contracts a divisor $E_f$ then $(E_f\cdot R)<0$.
Thus $E_f$ is not one of the $D_i$. Set
$Z_Y:=f(Z)$. If $f=(h^+)^{-1}\circ h$ is a
flip, set $Z_Y:= (h^+)^{-1}\bigl(h(Z)\bigr)$.
In both cases, as $D_i\cdot R=0$, we know that $Z_Y\subset \cap_{i\in J}f_*(D_i)$ and
$Z_Y\subset \ex(f^{-1})$.
Thus $\cap_{i\in J}f_*(D_i)$ is nonempty and it is not the $f_*$-image of a
non-contracted lc center. However, $\cap_{i\in J}f_*(D_i)$ is
a union of lc centers by \cite{Ambro03}, a contradiction.
\qed
\medskip

Using Theorem \ref{subcomp.collapse.thm} and Lemma \ref{numtrov.div.lem}
 for every step of an MMP
gives the following.

\begin{cor}\label{some.mmp.collapses.cor}
 Let $(X,\Delta)$ be dlt and
 $g:X\to Z$ a morphism.
 Let  $f \colon X \dasharrow Y$  be a birational map obtained
by running an $(X,\Delta)$-MMP over $Z$. 
Set $\Delta_Y := f_*\Delta$. 
Assume that  
there is a
numerically $g$-trivial effective divisor
whose support equals $\Delta^{=1}$.
Then  $\D(\Delta^{=1})$ collapses to
 $\D(\Delta_Y^{=1})$. \qed
\end{cor}

The next lemma gives other examples where the assumptions of
Theorem \ref{subcomp.collapse.thm}
are satisfied.

\begin{lem} \label{isolated.Qf.mmp.lem}
Let $X, Y$ be  normal, $\mathbb Q$-factorial
varieties and  $p:Y\to X$  a projective, birational morphism.
Let $\Gamma$ be a boundary on $Y$ such that 
$(Y, \Gamma)$ is lc.
Let  $f \colon Y \dasharrow Y_1$  be a 
  divisorial contraction or  flip
corresponding to a $(K_Y+\Gamma)$-negative extremal ray $R$ over $X$.
Then 
\begin{enumerate}\setcounter{enumi}{2}
\item either there is an $E_R\subset \ex(p)$ such that  $(E_R\cdot R)>0$
\item or $f$  contracts a divisor $E_f\subset \ex(p)$
and $Y_1\to X$ is a local isomorphism at the generic point of
$f(E_f)$.
\end{enumerate}
\end{lem}

Proof. Let  $p_1:Y_1\to X$ be the natural morphism.
Since $X$ is $\mathbb Q$-factorial,
$\ex(p_1)$ has pure codimension 1.

Assume first that $f$ is a divisorial contraction of
a divisor $E_f$. We are done if (4) holds. Otherwise
 $f(E_f)\subset \ex(p_1)$ hence there is an irreducible divisor
$E_{1}\subset \ex(p_1)$ that contains $f(E_f)$.
Hence its birational transform $E_R:=f^{-1}_*E_{1}$
has positive intersection with $R$.

Next assume that $f=(h^+)^{-1}\circ h$ is a flip.
Then $\ex(h^+)\subset \ex(p_1)$.
Since $X$ is $\mathbb Q$-factorial, there is an
effective anti-ample divisor supported on $\ex(p_1)$ \cite[2.62]{KM98}.
So there is  an irreducible divisor
$E_{1}\subset \ex(p_1)$ that has negative intersection with a flipped curve.
Thus its birational transform $E_R:=f^{-1}_*E_{1}$
has positive intersection with $R$. \qed
\medskip

As a consequence, we get the following
generalization of (\ref{DMR-thm-isolated}.3).

\begin{cor}  \label{isolated.Qf.mmp.cor}
Let $(X, \Delta)$ be a $\mathbb Q$-factorial pair and
 $(0\in X)$ a point such that 
$\bigl(X\setminus\{0\}, \Delta|_{X\setminus\{0\}}\bigr)$ is klt.
 Let  $p:Y\to X$ be a projective, birational morphism
that is an isomorphism over $X\setminus\{0\}$.
Set $E:=\ex (p)$ and assume that $(Y, E+p^{-1}_*\Delta)$ is  
 dlt. Then 
\begin{enumerate}
\item $\D(E)$  collapses to $\DMR(0\in X)$ and
\item if $(X, \Delta)$ is klt then $\D(E)$ is collapsible.
\end{enumerate}
\end{cor}

Proof. As we note in Definition \ref{Qf.defn}, we may assume that
$Y$ is $\mathbb Q$-factorial.
We run the $(Y,E+p^{-1}_*\Delta)$-MMP over $X$ to get
$$
(Y,E)=(Y_0,E_0)\stackrel{f_0}{\map} (Y_1,E_1)\stackrel{f_1}{\map}\cdots 
\stackrel{f_{r-1}}{\map} (Y_r,E_r)\stackrel{q}{\to} X.
$$
If $(X, \Delta)$ is not klt then  $(Y_r,E_r) $ is a dlt modification.
If $(X, \Delta)$ is klt then we let  $q:(Y_r,E_r)\to X $
denote the last extremal contraction.

By Lemma \ref{isolated.Qf.mmp.lem}, we can use Theorem
\ref{subcomp.collapse.thm} to conclude that
$\D(E_i)$ collapses to $\D(E_{i+1})$ for $i<r$,
proving the first claim.

 If $(X, \Delta)$ is klt then $q$ 
 contracts a single divisor $E_r$ to
a point. Thus $\D(E_r)$ is a point which shows (2). \qed
\medskip

During the proof of Theorem \ref{subcomp.collapse.thm}
 we have used  the following
variant of the  connectedness theorem. It is a special case of \cite[6.6]{Ambro03}. We include a proof here for reader's convenience.

\begin{prop}\label{p-connect}
Let $(X,\Delta)$ be a dlt pair and $f\colon X \to Y$  a projective morphism such that $f_*\mathcal{O}_X=\mathcal{O}_Y$ and $-(K_X+\Delta)$ is $f$-ample. Then for any $y\in Y$ there is at most one minimal lc center $Z$ of $(X,\Delta)$ that intersects $f^{-1}(y)$.  
\end{prop}

\begin{proof} 
This  follows from  \cite[Theorem 10]{Kollar11}
 by adding a general ample $\mathbb{Q}$-divisor $H\sim _{\mathbb{Q},f} -(K_X+\Delta)$ with small coefficients.

A more direct proof is the following.
By the  connectedness theorem \cite[Thm.5.48]{KM98},
 we know that if $Z_1$ and $Z_2$ are two lc centers that are minimal among all lc centers intersecting $f^{-1}(y)$, then we can find a chain $W_1,\dots,W_k$ of divisors in $\Delta^{=1}$ such that $Z_1\subset W_1$, $Z_2\subset W_k$, and $W_i \cap W_{i+1} \cap f^{-1}(y) \ne \emptyset$ for all $1 \le i \le k-1$.
Applying induction to $f\colon (W_1,{\rm Diff}_{W_1}\Delta)\to f(W_1)$, we conclude that $Z_1\subset W_1\cap W_2$, which implies that $Z_1\subset W_2$. Repeating the above argument for $f\colon (W_i,{\rm Diff}_{W_i}\Delta)\to f(W_i)$, we eventually obtain that $Z_1\subset W_k $. Then again by induction, we conclude that $Z_1=Z_2$.  
\end{proof}

\section{Proofs  of the main theorems}\label{pf.sect}

We prove Theorem \ref{t-sdr} and its consequences
 in this section. The strategy is to run a suitable MMP 
 terminating with a 
 dlt modification of $(X,\Delta)$ and  show that each step of the program 
corresponds to a collapse of the dual complex.
First we need to deal with divisors that are not $\mathbb{Q}$-Cartier.

\begin{defn}\label{Qf.defn} Let $X$ be a normal variety.
A {\it small, $\mathbb{Q}$-factorial modification}
is a proper, birational morphism $\pi:X^{\rm qf}\to X$ such that
$\pi$ is small, that is, there are no exceptional divisors,
and $X^{\rm qf}$ is $\mathbb{Q}$-factorial.

For a divisor $\Delta$ on $X$ set $\Delta^{\rm qf}:=\pi^{-1}_*\Delta$.
Note that if $(X, \Delta)$ is dlt  then
so is $\bigl(X^{\rm qf}, \Delta^{\rm qf}\bigr) $ and
$\D\bigl(\Delta^{=1}\bigr)=\D\bigl((\Delta^{\rm qf})^{=1}\bigr)$.

A dlt pair $(X, \Delta)$  always has projective,
small, $\mathbb{Q}$-factorial modifications
but not all lc pairs have them; see \cite[1.37]{kk-singbook}.

Let $\pi_1:\bigl(X_1, \Delta_1\bigr)\to (X,\Delta)$ be a projective,
small, $\mathbb{Q}$-factorial modification of a 
dlt pair $(X, \Delta)$.
Next run an $\bigl(X_1, \Delta_1-\Delta_1^{=1}\bigr)$-MMP over $X$
to get  $\pi_2:\bigl(X_2, \Delta_2\bigr)\to (X,\Delta)$
such that 
$$
-\Delta_2^{=1}\sim_{{\mathbf Q},\pi_2} 
K_{X_2}+\Delta_2-\Delta_2^{=1}
$$
 is $\pi_2$-nef. Thus
 $\pi_2:\bigl(X_2, \Delta_2\bigr)\to (X,\Delta)$
is a projective,
small, dlt, $\mathbb{Q}$-factorial modification such that 
$\supp \Delta_2^{=1}=\pi_2^{-1}(\supp\Delta^{=1})$.
\end{defn}

\begin{say}[Proof of Theorem~\ref{t-sdr}] \label{pf.of.t-sdr}
Let $\pi_2:\bigl(X_2, \Delta_2\bigr)\to (X,\Delta)$ be as above.

After a suitable sequence of blow-ups we get $Y_2\to Y$ such that
the induced rational map $g_2:Y_2\map X_2$ is a morphism. 
As noted in Paragraph \ref{bu.and.baryc.say},
 $\D\bigl(\supp g_2^{-1}(\Delta_2^{=1})\bigr)$
is simple-homotopy equivalent to 
$\D(E)=\D\bigl(\supp g^{-1}(\Delta^{=1})\bigr)$.
By replacing $X$ by $X_2$ and $Y$ by $Y_2$,
 we may assume from now on that $X$ is $\mathbb{Q}$-factorial.

Consider the boundary
$$
\Gamma := g^{-1}_*\Delta^{<1} + E + \textstyle{\sum} b_i F_i
$$
where the sum on the right hand side is taken over all 
 $g$-exceptional divisors $F_i$ that are not contained in the support of $E$
 and $b_i := \max\{\frac{1 - a(F_i,X,\Delta)}{2},\frac{1}{2}\}$. 
We run a $(K_Y,\Gamma)$-MMP  over $X$ as in  \cite{BCHM10}.

By construction, the $\mathbb Q$-divisor
$$
G := K_Y+ \Gamma - g^*(K_X+\Delta)
$$
is effective and its support consists precisely of all the $g$-exceptional 
divisors with positive log discrepancy with respect to $(X,\Delta)$. As the push forward of $G$ to $\tilde{X}$ is effective, exceptional and semiample over $X$, it must be trivial by negativity. This means that the map $Y \dasharrow \tilde X$ produced by the MMP contracts all components of $G$. It follows therefore that if we denote by $\tilde{\Delta}$ the birational transform of $\Gamma$ on $\tilde X$, then $(\tilde X, \tilde \Delta)$ is a
 dlt modification of $(X,\Delta)$.
Thus  $\D(\Delta^{=1})$ is PL homeomorphic to $\D(\tilde\Delta^{=1})$ by 
Proposition \ref{crep.bir.dlt.models}. 

Furthermore, since $E=\supp g^*(\Delta^{=1})$
and $g^*(\Delta^{=1}) $ is numerically $g$-trivial, 
Corollary \ref{some.mmp.collapses.cor} implies that 
$\D(E)$ collapses to $\D(\tilde\Delta^{=1})$.\qed
\end{say}

We can now state and prove the general form of Theorem \ref{DMR-thm-isolated}

\begin{thm}\label{DMR-thm}
Let $X$ be a normal variety and $\Delta$ an effective boundary on $X$
such that $K_X+\Delta$ is $\mathbb Q$-Cartier.
Let $W\neq \emptyset$ be the non-klt locus of $(X,\Delta)$.
Let $p:Y\to X$ be a proper, birational morphism and 
 $E$ the divisorial part of $\supp p^{-1}(W)$. Assume that 
  $\bigl(Y,E+p^{-1}_*\Delta^{<1}\bigr)$ is dlt.
\begin{enumerate}
\item If $K_Y+E+p^{-1}_*\Delta^{<1}$ is $p$-nef then
the dual complex $\DMR(X,\Delta ):=\D(E)$ is independent of $p$, up-to
PL homeomorphism.
\item If  $E=\supp p^{-1}(W)$ then
 the dual complex $\D(E)$ is
 simple-homotopy equivalent to
$\DMR( X,\Delta)$.
\item If  $X$ is $\mathbb Q$-factorial 
and $p$ is an isomorphism over $X\setminus W$   
then  $\D(E)$ collapses to $\DMR(X,\Delta)$.
\end{enumerate}
\end{thm}

Proof. Part (1) is noted in Definition \ref{DMR.defn}.

In order to see (2) we use Lemma \ref{3.3.1.4.HMX12.lem}
to obtain a $\mathbb Q$-factorial dlt modification
$q: \bigl(X^{\rm dlt}, \Delta^{\rm dlt}\bigr)\to (X,\Delta )$
 such that
$(\Delta^{\rm dlt})^{=1}=\supp q^{-1}(W)$. 
After some blow-ups we get
$\pi:Y_1\to Y$ such that $q^{-1}\circ p\circ \pi:Y_1\to X^{\rm dlt}$
is a morphism and $\D\bigl(\supp (p\circ \pi)^{-1}(W)\bigr)$
is simple-homotopy equivalent to  $\D(E)$.

By Theorem \ref{t-sdr}, $\D\bigl(\supp (p\circ \pi)^{-1}(W)\bigr)$
collapses to $ \D\bigl((\Delta^{\rm dlt})^{=1}\bigr)=\DMR(X,\Delta )$.

Finally assume that $X$ is $\mathbb Q$-factorial 
and $p$ is an isomorphism over $X\setminus W$.
Let $f_i:Y_i\to Y_{i+1}$ be an extremal contraction
of a divisor $E_i\subset Y_i$ in our MMP and $p_{i+1}:Y_{i+1}\to X$ the
projection. Note that $p_{i+1}\bigl(f_i(E_i)\bigr)\subset W$,
thus $p_{i+1}$ is not a local isomorphism at the generic point of
$f_i(E_i) $. Thus the first alternative of Lemma \ref{isolated.Qf.mmp.lem}
applies at each step of the MMP, hence  Theorem \ref{subcomp.collapse.thm}
implies that $\D(E)$ collapses to $\DMR(X,\Delta)$.
\qed
\medskip

The following argument is similar to \cite[3.3.1.4]{HMX12}.

\begin{lem}\label{3.3.1.4.HMX12.lem}
Let $(X,\Delta)$ be a pair such that $K_X+\Delta$ is
$\mathbb{Q}$-Cartier.  Let $W$  be the non-klt locus 
of $(X,\Delta)$. Then there
exists a dlt modification $p:Y\to X$ such that ${\rm Supp} (p^{-1}W)$ is a
divisor.
\end{lem}
\begin{proof} Let $p^{\rm lc}:(X^{\rm lc},\Delta^{\rm lc})\to (X,\Delta)$ be
the log canonical modification as in Definition \ref{dlt-mod.defn}.
As we noted, $\ex(p^{\rm lc})$ has pure codimension 1
and it is contained in $\supp\Bigl((\Delta^{\rm lc})^{=1}\bigr)$. Thus
$W^{\rm lc}$, the non-klt locus of  $(X^{\rm lc},\Delta^{\rm lc}) $, 
is equal to  the preimage of $W$.

Since any dlt modification of $(X^{\rm lc}, \Delta^{\rm lc})$ is also  a dlt
modification of $(X,\Delta)$,  it suffices to show that
there is a dlt modification of $q:Y\to (X^{\rm lc},\Delta^{\rm lc})$ such that
${\rm Supp}(q^{-1}(W^{\rm lc}))$ is of pure codimension 1. 
 
Consider an arbitrary dlt modification 
$q_1: Y_1\to (X^{\rm lc}, \Delta^{\rm lc})$ and write 
$$q_1^*(K_{X^{\rm lc}}+\Delta^{\rm lc})\sim_{\mathbb Q}K_{Y_1}+\Gamma_1$$
then $ W^{\rm lc}=q_1(\Gamma^{=1}_1)$. It follows from \cite{BCHM10} that there
is a minimal model $p_2: Y_2\to X^{\rm lc}$ of $(Y_1, \Gamma^{<1}_1)$ over
$X^{\rm lc}$. 
Denote the push forward of  $\Gamma_1$ on $Y_2$ to be $\Gamma_2$.
Thus $-\Gamma_2^{=1}\sim_{\mathbb{Q}, X^{\rm lc}} K_{Y_2}+\Gamma^{<1}_2$
is nef, which implies that 
$${\rm Supp}(p_2^{-1}(W^{\rm lc}))={\rm Supp}(p_2^{-1}p_2 (\Gamma_2^{=1}))={\rm
Supp}(\Gamma_2^{=1})$$ has the same support as $\Gamma^{=1}_2$.

Though $(Y_2,\Gamma_2)$ is only log canonical, we can take any dlt modification
$r: Y\to(Y_2,\Gamma_2)$  and set $q=p_2\circ r$. 
Then ${\rm Supp}(q^{-1}W^{\rm lc})={\rm Supp}(r^*\Gamma^{=1}_2)$  is of pure
codimension 1 since $\Gamma^{=1}_2$ is $\mathbb{Q}$-Cartier. 
\end{proof}

\begin{say}[Proof of Theorem \ref{c-dlt}] 
Since a dlt pair is always the limit of klt pairs (see \cite[2.43]{KM98}), we can assume that $0\in (X,\Delta)$ is a point on a klt pair. 
As in  \cite[Lem.2.5]{HM06}
 there exists a $\mathbb{Q}$-divisor $H$ such that  $(X,\Delta+H)$ is klt on $X\setminus \{0\}$ and there is precisely  one divisor $E$ with discrepancy $-1$ over $ 0$. 
Thus $\DMR(0\in X,\Delta+H)$ is a single point. 

By (\ref{DMR-thm}.2)
$\DR(0\in X)$ is simple-homotopy equivalent to $\DMR(0\in X,\Delta+H)$
hence to a point.
 \qed
\end{say}

\begin{say}[Versions over nonclosed fields]\label{Gal-version.rem}
Let $k$ be a field of characteristic 0,
$X$ a normal variety defined over $k$ and
$0\in X$ a $k$-point. Let $f:X'\to X$ be a resolution
defined over $k$ such that $E':=\supp f^{-1}(0)$ is a
simple normal crossing divisor defined  over $k$. 
(See \cite[Defn.1.7]{kk-singbook} for the correct definition
of a simple normal crossing divisor.)
Set $\D(E'):=\D\bigl(E'_{\bar k}\bigr)$ where
$\bar k\supset k$ is an algebraic closure.
The Galois group $\operatorname{Gal}(\bar k/k)$ acts on
$\D(E')$ and the usual arguments show that 
$\DR(0\in X)$ is well defined up-to
$\operatorname{Gal}(\bar k/k)$-equivariant
simple-homotopy equivalence.

One needs to check that the collapses in Theorem \ref{subcomp.collapse.thm}
can be done equivariantly. Let $D_0$ be a divisor defined over $k$
such that $(D_0\cdot R)>0$. Over $\bar k$, it can break up into a
collection of disjoint divisors  $D_0^i$. The contraction
$h:X\to X_1$ of the ray $R$ becomes the contraction
of a face $F$ over $\bar k$, but $D_0$ is 
strictly negative on $F\setminus\{0\}$.
The proof of  Theorem \ref{subcomp.collapse.thm} now shows
that each contracted lc center over $\bar k$ intersects
exactly one of the $D_0^i$. Thus the collapses prescribed by
the different $D_0^i$ occur in disjoint sets so they can be performed
simultaneously.

The proofs of Theorems \ref{c-dlt}--\ref{t-RC}  apply without changes to
yield the following generalizations.
\end{say}

\begin{thm}\label{c-dlt-gal} 
Let $0\in X$ be a $k$-point on  a dlt pair $(X,\Delta)$ defined over $k$.
Then $\DR(0\in X)$ is $\operatorname{Gal}(\bar k/k)$-equivariantly
contractible. \qed
\end{thm}

\begin{thm}\label{t-RC-gal}
Let $0\in C$ be a $k$-point on a smooth curve
and $f\colon X\to (0\in C)$ a flat, projective morphism.
 Assume that the general geometric fibers $X_t$ are smooth, rationally connected
and $F_0:=\red f^{-1}(0)$ is a simple normal crossing divisor. Then 
  $\D(F_0)$ is $\operatorname{Gal}(\bar k/k)$-equivariantly contractible. \qed
\end{thm}

Note that a contractible space with a finite group action need not be
equivariantly contractible (cf.\ \cite{Flo-Ric}),
 thus,  in this respect, the dual complexes coming from
algebraic geometry  behave better than arbitrary
simplicial complexes.

\section{Quotients of  dlt pairs}\label{sec.quot.dlt}

\begin{prop}\label{qdlt.prop}
Let $(x\in X)$ be the spectrum of a $d$-dimensional local ring
and $\Delta$ an effective divisor such that $(X,\Delta)$ is lc
and $x$ is an lc center of $(X, \Delta)$. The following are equivalent.
\begin{enumerate}
\item There are ${\mathbb Q}$-Cartier divisors
$D_1, \dots, D_d\subset \Delta^{=1}$
such that $x\in D_i$. 
\item There is a semi-local, snc pair
$\bigl(x'\in X', D'_1+ \dots + D'_d\bigr)$
and an Abelian group $G$ acting on it  (preserving each of the $D'_i$)
such that
$$
(x\in X, \Delta)=\bigl(x\in X, D_1+ \dots + D_d\bigr)\cong
\bigl(x'\in X', D'_1+ \dots + D'_d\bigr)/G.
$$
\end{enumerate}
\end{prop}

Proof. The implication (2) $\Rightarrow$ (1) is well known;
cf.\ \cite[Prop.5.20]{KM98}.

To  converse is outlined in \cite[Sec.18]{K-etal92}.
We construct $\pi:X'\to X$ as follows.
By assumption, for every $i$ there is an
$m_i>0$ such that $m_iD_i\sim 0$. These give degree $m_i$ cyclic covers
$X'_i\to X$; let $\pi:X'\to X$ be their composite. Then
$X'\to X$ is Galois with group $\prod_i {\mathbb Z}/m_i$
and it branches only along the $D_i$.
Set $D'_i:=\operatorname{red} \pi^{-1}(D_i)$. Then
$\bigl(X', D'_1+ \dots + D'_d\bigr) $ is lc by
\cite[Prop.5.20]{KM98}.

In general $x':=\pi^{-1}(x)$ may consist of several points.
At each of them, the $D'_i$ are Cartier. We claim that in  fact
$X'$ and the $D'_i$ are smooth. This is proved by induction on the dimension.
The $d=1$ case is clear.

By adjunction,
$\bigl(D'_d, D'_1|_{D'_d}+ \dots + D'_{d-1}|_{D'_d}\bigr) $ is lc,
thus $D'_d$ is smooth by induction. Since $D'_d$ is a Cartier divisor,
this implies that $X'$ is smooth. \qed

\begin{defn} \label{qdlt.defn}
A log canonical pair  $(X,\Delta)$ is called {\it quotient-dlt},
abbreviated as {\it qdlt,} if for every
lc center $Z\subset X$  the local scheme
$$
\bigl(\spec \mathcal{O}_{Z,X}, \Delta|_{  \spec \mathcal{O}_{Z,X}}\bigr)
$$
satisfies the equivalent conditions of Proposition \ref{qdlt.prop}.
\end{defn}
It is clear that being qdlt is preserved by any
$(X, \Delta)$-MMP.
It follows from Definition \ref{dual.cpx.defn}, that we can define
$\D(\Delta^{=1})$ for a qdlt pair. It seems to us that, for
dual-complex problems, qdlt pairs form the most general class
where the main results hold.
We need the following lemma on extending qdlt modifications.

\begin{lem}\label{qdlt.rxtend.lem}
  Let $(X,\Delta)$ be a quasi-projective, lc pair.
Let $U\subset X$ be an open subset such that no lc center
is contained in $X\setminus U$.
Let $g_U: U'\to U$ be a quasi-projective, qdlt modification.
Then there is a quasi-projective, qdlt modification 
 $g_X: X'\to X$ extending $g_U$.
\end{lem}

Note that since  no lc center
is contained in $X\setminus U$, 
every exceptional divisor of $g_X$ intersects $U'$.
\medskip

Proof. Let $\bar{X}$  be a normal compactification of $U'$ that  is
projective over $X$. Let $Y$ be a projective resolution of $\bar{X}$
such that the preimage of $\supp \Delta_{U'}\cup\bigl(\bar{X}\setminus U'\bigr)$
is a simple normal crossing divisor.

Denote by $p:Y\to X$ the composite morphism. For  $p$-exceptional
divisors $E_i$ let $b_i=\max \{\frac{1}{2},\frac{1-a(E_i,
X,\Delta)}{2}\}$ and set
$\Delta_Y=p_*^{-1}\Delta+\sum_i b_iE_i $. Applying \cite{HX12}, we obtain
that $(Y,\Delta_Y)$ has a log canonical model $X''$ over $\bar{X}$ that
is a compactification of  $U'$. 

Let $\Delta''$  be the push forward of $\Delta_Y$. Then 
$(X'',\Delta'')$ does not have log canonical centers contained in
$X''\setminus U'$ hence $(X'',\Delta'')$ is qdlt. Furthermore, 
 $K_{X''}+\Delta''-h^*(K_{X}+\Delta)$
 is effective and its support is the same as the divisorial part of
$X''\setminus U'$ which we denote by $\sum_{i\in I}E_i$. We conclude that
$\{E_i\colon i\in I\}$ are precisely the divisorial components in the stable
base locus $${\bf B}(X''/X, K_{X''}+\Delta'')={\bf
B}_{-}(X''/X,K_{X''}+\Delta'').$$
That is, if  $H''$ is an ample divisor on $X''$
 there is a $0<\epsilon\ll 1$ such that all the 
$E_i$ are also contained in ${\bf
B}(X''/X, K_{X''}+\Delta''+\epsilon H'')$.

By the method of \cite[2.43]{KM98}, 
 there is a divisor $\Delta''_{\epsilon}$ on $X''$ such
that  $\Delta''_{\epsilon}\sim_{\mathbb{Q}}K_{X''}+\Delta''+\epsilon H'' $
and $(X'',\Delta''_{\epsilon})$ is klt.

 By \cite{BCHM10} a suitable $(X'',\Delta''_{\epsilon})$-MMP  over $X$
 terminates with a minimal model $X'$. Then the rational map
$\phi:X''\dasharrow X'$ is an isomorphism on $U'$ and 
it contracts all the $E_i$. Set $\Delta':=\phi_*\Delta''$.
Then $K_{X'}+\Delta'\sim_{\mathbb{Q},X}0$ and  $X'\setminus U'$
has codimension $\geq 2$. Therefore
$g_X:X'\to X$ is a  quasi-projective, qdlt modification extending $g_U$.
\qed
\medskip

This implies that  qdlt pairs have a 
toroidal dlt modification.

\begin{prop}\label{qdlt.toroidal.prop}
Let $(X,\Delta)$ be a  quasi projective, qdlt pair. 
 Let $U\subset X$ be an open set containing all the log canonical centers of $(X,\Delta)$ such that $\Delta|_U=\Delta^{=1}|_U$
and  $(U,\Delta_U)=(U,\Delta^{=1}|_U)$ is toroidal (such $U$ exists by definition).
Then there exists a dlt modification $g^{dlt}:(X^{dlt},\Delta^{dlt})\to (X,\Delta)$ such that over $U$, $(U', \Delta_{U'}):=(g^{-1}(U),\Delta^{dlt}|_{U'})$ is snc, 
$g^{dlt}|_{U'}: (U', \Delta_{U'})\to (U,\Delta_U)$
is toroidal and
 $\D((\Delta^{dlt})^{=1})$ is a subdivision of $\D(\Delta^{=1})$. 
\end{prop}

\begin{proof} Given the toroidal pair $(U,\Delta_U)$ by \cite{KKMSD73}
 there is a toroidal log resolution  $g_U:(U',\Delta_{U'}) \to (U,\Delta)$.
 Since it is a toroidal morphism, 
$g_U^*(K_U+\Delta_U)=K_{U'}+\Delta_{U'}$ and $\D(\Delta_{U'})$ is subdivision of $\D(\Delta_U)$. By Lemma \ref{qdlt.rxtend.lem}, we can extend
$g_U:U'\to U$ to $g_X:X'\to X$. \end{proof}

\begin{cor} \label{qdlt.toroidal.cor2}
Let $(X,\Delta_X)$ be a  quasi projective, qdlt pair
and $g: \bigl(Y, \Delta_Y\bigr)\to (X,\Delta_X)$
a dlt modification. Then
$\D\bigl(\Delta_X^{=1}\bigr)$ is PL homeomorphic to
$\D\bigl(\Delta_Y^{=1}\bigr)$. \qed
\end{cor}

Another good property of quotient-dlt singularities is that
the theorem on extracting one divisor
(cf.\ \cite[1.39]{kk-singbook}) 
that fails for dlt singularities
does hold for quotient-dlt singularities.

\begin{lem} Let $(X,\Delta)$ be a quasi-projective qdlt pair. Let $E$ be a divisor such that $a(E;X,\Delta)=-1$. Then there exists a model $f:Y\to X$ such that $E$ is the sole exceptional divisor of $f$ 
 and $(Y,f^{-1}_*\Delta+E)$ is qdlt. 
\end{lem}

\begin{proof} We may assume that $X$ is $\mathbb Q$-factorial.
We pick  $U$ as in Proposition \ref{qdlt.toroidal.prop} 
and let $g_U:U'\to U$ be the toroidal morphism  such that $\ex(g_U)=E$. 
 By Lemma \ref{qdlt.rxtend.lem}, we can extend
$g_U:U'\to U$ to $g_X:X'\to X$.\end{proof}

This following result is essentially in \cite[Sec.5]{HX09}.

\begin{prop}\label{gdlt.fano.contr.prop}
 Let $(X,\Delta_X)$ be a ${\mathbb Q}$-factorial
dlt or qdlt pair. Let 
$g\colon (X,\Delta_X)\to Y$ be a Fano contraction of an extremal ray.
Assume that $\Delta_X^{=1}$ is $g$-vertical.

Then $Y$ is ${\mathbb Q}$-factorial and there is a
${\mathbb Q}$-divisor $\Delta_Y$ such that
$(Y, \Delta_Y)$ is qdlt and
there is a natural identification
$\D\bigl(\Delta_X^{=1}\bigr)\cong \D\bigl(\Delta_Y^{=1}\bigr)$.
\end{prop}

\begin{proof}   It follows from \cite[3.36]{KM98} that $Y$ is $\mathbb{Q}$-factorial. Since  $\rho(Y/X)=1$, we know that each $g$-vertical component is mapped to a divisor on $Y$ and different $g$-vertical divisors are mapped to different divisors on $Y$ (see \cite[5.1, 5.2]{HX09}). 
Write $\Delta_X^{=1}=\sum_{i\in I} E_i$. Then $\Delta_Y=\sum_{i\in I} F_i$ where 
$F_i:=g(E_i)$  and $E_i=g^{-1}(F_i)$. 
Thus, for every $J\subset I$ we have
$$
g^{-1}(\cap_{i\in J}F_i )=\cap_{i\in J}g^{-1}(F_i)=
\cap_{i\in J} E_i,
$$
which shows that $\D\bigl(\sum_{i\in I} E_i\bigr)=
\D\bigl(\sum_{i\in I} F_i\bigr)$.

Similar to the argument of \cite[5.5]{HX09}, we see that 
the log canonical centers of $(Y,\Delta_Y)$ are precisely the nonempty intersections $\cap_{i\in J}F_i$ where $J\subset I$.  
In particular, $(Y,\Delta_Y)$ is qdlt.  
\end{proof}

\medskip

We can now prove the following strengthening of Theorem \ref{t-RC}.

\begin{thm}\label{t-rat}
Let $f:X\to (0\in C)$ be a flat proper family.   Assume that
general fibers $X_t$ are rationally connected
and $(X, \red f^{-1}(0))$ is qdlt.
Then $\D\bigl(\red f^{-1}(0)\bigr)$ is contractible. 
\end{thm}

\begin{proof} By  Theorem \ref{t-sdr} we may assume that $X$ is smooth,
$f$ is projective  and
$\Delta:=\red f^{-1}(0)$ is snc.

 Then we run a $(K_X+\Delta)$-MMP over $C$.
This is the same as running  a $(K_X+\Delta-\epsilon f^{-1}(0))$-MMP with 
 for  $1\gg \epsilon>0$. 
As $X_t$ is smooth and rationally connected,  $K_X$ is not pseudo-effective, 
 thus by \cite{BCHM10} a suitable MMP terminates with a Fano contraction 
$g:(X', \Delta')\to Y$ over $C$  where $\dim X'> \dim Y$. 

Since $f^{-1}(0)$ is numerically $f$-trivial and its support is 
$\Delta^{=1}=\Delta$, 
 Corollary \ref{some.mmp.collapses.cor} guarantees
that $\D(\Delta)$ collapses to $\D(\Delta')$.

If $Y=C$ then $\Delta'=\red (f')^{-1}(0)$ is irreducible
  since $\rho(X'/Y)=1$. Thus $\D(\Delta')$ is a point and we are done.

 If $\dim Y>1$ then we look at $f_Y:Y\to C$. 
By Proposition \ref{gdlt.fano.contr.prop},
$(Y, \red f_Y^{-1}(0))$ is again qdlt
and $\D\bigl(\red (f')^{-1}(0)\bigr)$ can be identified with
 $\D\bigl(\red f_Y^{-1}(0)\bigr)$. The latter is collapsible
by induction, hence so is $\D(\Delta')$ and hence $\D(\Delta)$.
  \end{proof}


\bigskip

\noindent TdF: Dept.\ of Mathematics, University of Utah, Salt Lake City, UT 84112, USA

{\begin{verbatim}defernex@math.utah.edu\end{verbatim}}
\medskip

\noindent JK: Princeton University, Princeton NJ 08544-1000

{\begin{verbatim}kollar@math.princeton.edu\end{verbatim}}
\medskip

\noindent  CX: Beijing International Center of Mathematics Research, 
Beijing, 100871, China  

Dept.\ of Mathematics, University of Utah, 
Salt Lake City, UT 84112, USA

{\begin{verbatim} cyxu@math.pku.edu.cn,  cyxu@math.utah.edu\end{verbatim}}


\begin{bibdiv}
\begin{biblist}

\bib{ABW11}{article}{
   author={Arapura, Donu},
   author={Bakhtary, Parsa},
   author={W{\l}odarczyk, Jaros{\l}aw},
   title={ Weights on cohomology, invariants of singularities, and dual complexes},
   journal={Math. Ann.},
   volume={357},
   date={2013},
   number={2},
   pages={513--550},
}


\bib{AKMW02}{article}{
   author={Abramovich, Dan},
   author={Karu, Kalle},
   author={Matsuki, Kenji},
   author={W{\l}odarczyk, Jaros{\l}aw},
   title={Torification and factorization of birational maps},
   journal={J. Amer. Math. Soc.},
   volume={15},
   date={2002},
   number={3},
   pages={531--572 (electronic)},
}

\bib{Ambro03}{article}{
    AUTHOR = {Ambro, Florin},
     TITLE = {Quasi-log varieties},
   JOURNAL = {Tr. Mat. Inst. Steklova},
  FJOURNAL = {Trudy Matematicheskogo Instituta Imeni V. A. Steklova.
              Rossi\u\i skaya Akademiya Nauk},
    VOLUME = {240},
      YEAR = {2003},
    NUMBER = {Biratsion. Geom. Linein. Sist. Konechno Porozhdennye Algebry},
     PAGES = {220--239},
      ISSN = {0371-9685},
   MRCLASS = {14E30 (14J10)},
  MRNUMBER = {MR1993751 (2004f:14027)},
MRREVIEWER = {Tomasz Szemberg},
}

\bib{Ambro04}{article}{
   author={Ambro, Florin},
   title={Shokurov's boundary property},
   journal={J. Differential Geom.},
   volume={67},
   date={2004},
   number={2},
   pages={229--255},
}

\bib{BCHM10}{article}{
   author={Birkar, Caucher},
   author={Cascini, Paolo},
   author={Hacon, Christopher D.},
   author={McKernan, James},
   title={Existence of minimal models for varieties of log general type},
   journal={J. Amer. Math. Soc.},
   volume={23},
   date={2010},
   number={2},
   pages={405--468},
}

\bib{Coh73}{book}{
    AUTHOR = {Cohen, Marshall M.},
     TITLE = {A course in simple-homotopy theory},
      NOTE = {Graduate Texts in Mathematics, Vol. 10},
 PUBLISHER = {Springer-Verlag},
   ADDRESS = {New York},
      YEAR = {1973},
     PAGES = {x+144},
   MRCLASS = {57C10},
  MRNUMBER = {0362320 (50 \#14762)},
MRREVIEWER = {R. M. F. Moss},
}

\bib{Corti07}{collection}{
   title={Flips for 3-folds and 4-folds},
   series={Oxford Lecture Series in Mathematics and its Applications},
   volume={35},
   editor={Corti, Alessio},
   publisher={Oxford University Press},
   place={Oxford},
   date={2007},
   pages={x+189},
}

\bib{Deligne71}{article}{
   author={Deligne, Pierre},
   title={Th\'eorie de Hodge. II},
    journal={Inst. Hautes \'Etudes Sci. Publ. Math.},
   number={40},
   date={1971},
   pages={5--57},
}

\bib{dF12}{article}{
AUTHOR = {de Fernex, Tommaso },
TITLE = {problem session at AIM workshop 
``ACC for minimal log discrepancies and termination of flips''},
YEAR = {2012},
}





\bib{FM83}{book}{
    EDITOR = {Friedman, Robert},
 EDITOR = {Morrison, David R.},
    TITLE = {The birational geometry of degenerations},
    SERIES = {Progr. Math.},
    VOLUME = {29},
     PAGES = {1--32},
 PUBLISHER = {Birkh\"auser Boston},
   ADDRESS = {Mass.},
      YEAR = {1983},
   MRCLASS = {14J10 (14E05 32J15)},
  MRNUMBER = {690262 (84g:14032)},
MRREVIEWER = {Masa-Nori Ishida},
}

\bib{Flo-Ric}{article}{
   author={Floyd, E.},
author={Richardson, R.W.},
   title={ An action of a finite group on an n-cell without stationary points},
    journal={Bull. Amer. Math. Soc.},
   number={65},
   date={1959},
   pages={73--76},
}

\bib{FujTak-c3}{article}{
   author = {Fujino, Osamu},
 author ={Takagi, Shunsuke},
    title = {On the F-purity of isolated log canonical singularities},
  journal = {Compos. Math. (to appear)},
   eprint = {1112.2383},
 keywords = {Mathematics - Algebraic Geometry, Mathematics - Commutative Algebra, Primary 14B05, Secondary 13A35, 14E30},
     year = {2011},
   adsurl = {http://adsabs.harvard.edu/abs/2011arXiv1112.2383F},
  adsnote = {Provided by the SAO/NASA Astrophysics Data System},
}

\bib{Fuj-c4}{article}{
    AUTHOR = {Fujino, Osamu},
     TITLE = {Abundance theorem for semi log canonical threefolds},
   JOURNAL = {Duke Math. J.},
  FJOURNAL = {Duke Mathematical Journal},
    VOLUME = {102},
      YEAR = {2000},
    NUMBER = {3},
     PAGES = {513--532},
      ISSN = {0012-7094},
     CODEN = {DUMJAO},
   MRCLASS = {14E30 (14C20 14E07)},
  MRNUMBER = {1756108 (2001c:14032)},
MRREVIEWER = {Tomasz Szemberg},
}

\bib{Fuj-c2}{article}{
    AUTHOR = {Fujino, Osamu},
     TITLE = {The indices of log canonical singularities},
   JOURNAL = {Amer. J. Math.},
  FJOURNAL = {American Journal of Mathematics},
    VOLUME = {123},
      YEAR = {2001},
    NUMBER = {2},
     PAGES = {229--253},
      ISSN = {0002-9327},
     CODEN = {AJMAAN},
   MRCLASS = {14B05 (14E05 14L30)},
  MRNUMBER = {MR1828222 (2002c:14003)},
MRREVIEWER = {Shihoko Ishii},
}

\bib{Fujino}{article}{
    AUTHOR = {Fujino, Osamu},
     TITLE = {What is log terminal?},
 BOOKTITLE = {Flips for 3-folds and 4-folds},
    SERIES = {Oxford Lecture Ser. Math. Appl.},
    VOLUME = {35},
     PAGES = {49--62},
 PUBLISHER = {Oxford Univ. Press},
   ADDRESS = {Oxford},
      YEAR = {2007},
   MRCLASS = {14E30},
  MRNUMBER = {MR2359341},
}

\bib{Fuj-c1}{article}{
   author = {Fujino, Osamu},
    title = {On isolated log canonical singularities with index one},
journal = {J. Math. Sci. Univ. Tokyo}
   year = {2011},
volume = {18},
}

\bib{Fuj}{article}{
AUTHOR = {Fujino, Osamu},
     TITLE = {Fundamental theorems for the log minimal model program},
   JOURNAL = {Publ. Res. Inst. Math. Sci.},
  FJOURNAL = {Publications of the Research Institute for Mathematical
              Sciences},
    VOLUME = {47},
      YEAR = {2011},
    NUMBER = {3},
     PAGES = {727--789},
}
	

\bib{Hatcher}{book}{
author={Hatcher, Allen}
title={Algebraic topology}
PUBLISHER = {Cambridge Univ. Press},
date={2002},
}

\bib{HM06}{article}{
   author={Hacon, Christopher D.},
   author={McKernan, James},
   title={Boundedness of pluricanonical maps of varieties of general type},
   journal={Invent. Math.},
   volume={166},
   date={2006},
   number={1},
   pages={1--25},
}

\bib{HMX12}{article}{
   author={Hacon, Christopher D.},
   author={McKernan, James},
   author={Xu, Chenyang},
   title={ACC for log canonical thresholds},
   journal={Ann. of Math. (to appear)}
   note={arXiv:1208.4150},
   date={2012}
}

\bib{HX09}{article}{
   author={Hogadi, Amit},
   author={Xu, Chenyang},
   title={Degenerations of rationally connected varieties},
   journal={Trans. Amer. Math. Soc.},
   volume={361},
   date={2009},
   number={7},
   pages={3931--3949},
}

\bib{HX12}{article}{
   author={Hacon, Christopher D.},
   author={Xu, Chenyang},
   title={Existence of log canonical closures},
   journal={Invent. Math.},
   volume={192},
   date={2013},
   number={1},
   pages={161--195},
  date={2013}
}




\bib{KK10}{article}{
   author={Koll{\'a}r, J{\'a}nos},
   author={Kov{\'a}cs, S{\'a}ndor J.},
   title={Log canonical singularities are Du Bois},
   journal={J. Amer. Math. Soc.},
   volume={23},
   date={2010},
   number={3},
   pages={791--813},
}

\bib{KK11}{article}{
   author = {Kapovich, Michael},
 author = {Koll{\'a}r, J{\'a}nos},
    title = {Fundamental groups of links of isolated singularities},
  journal = {Jour. AMS  (to appear)},
     date={2011},
 eprint = {1109.4047},
}

\bib{KKMSD73}{book}{
  author={Kempf, G.},
  author={Knudsen, F.F.},
  author={Mumford, D.},
  author={Saint-Donat, B.},
  title={Toroidal embeddings. I},
  series={Lecture Notes in Mathematics, Vol. 339},
  publisher={Springer-Verlag},
  place={Berlin},
  date={1973},
  pages={viii+209},
}



\bib{KM98}{book}{
   author={Koll{\'a}r, J{\'a}nos},
   author={Mori, Shigefumi},
   title={Birational geometry of algebraic varieties},
   series={Cambridge Tracts in Mathematics},
   volume={134},
   note={With the collaboration of C. H. Clemens and A. Corti;
   Translated from the 1998 Japanese original},
   publisher={Cambridge University Press},
   place={Cambridge},
   date={1998},
   pages={viii+254},
   
   
}


\bib{MR1284817}{article}{
    AUTHOR = {Keel, Sean},
  AUTHOR = {Matsuki, Kenji},
 AUTHOR = {McKernan, James},
     TITLE = {Log abundance theorem for threefolds},
   JOURNAL = {Duke Math. J.},
  FJOURNAL = {Duke Mathematical Journal},
    VOLUME = {75},
      YEAR = {1994},
    NUMBER = {1},
     PAGES = {99--119},
      ISSN = {0012-7094},
     CODEN = {DUMJAO},
   MRCLASS = {14E30 (14J30 14J35)},
  MRNUMBER = {MR1284817 (95g:14021)},
MRREVIEWER = {Mark Gross},
}

\bib{K-etal92}{book}{
    AUTHOR = {Koll{\'a}r, J{\'a}nos},
    TITLE =  {Flips and abundance for algebraic threefolds},
       NOTE = {Papers from the Second Summer Seminar on Algebraic Geometry
              held at the University of Utah, Salt Lake City, Utah, August
              1991,
              Ast\'erisque No. 211 (1992)},
 PUBLISHER = {Soci\'et\'e Math\'ematique de France},
       YEAR = {1992},
     PAGES = {115--126},
      ISSN = {0303-1179},
   MRCLASS = {14E30 (14E35 14M10)},
  MRNUMBER = {94f:14013},
MRREVIEWER = {Mark Gross},
}


\bib{Koll93}{article}{
   author={Koll{\'a}r, J{\'a}nos},
   title={Shafarevich maps and plurigenera of algebraic varieties},
   journal={Invent. Math.},
   volume={113},
   date={1993},
   number={1},
   pages={177--215},
}

\bib{Kollar07}{article}{
   author={Koll{\'a}r, J{\'a}nos},
   title={A conjecture of Ax and degenerations of Fano varieties},
   journal={Israel J. Math.},
   volume={162},
   date={2007},
   pages={235--251},
}

\bib{k-res}{book}{
    AUTHOR = {Koll{\'a}r, J{\'a}nos},
     TITLE = {Lectures on resolution of singularities},
    SERIES = {Annals of Mathematics Studies},
    VOLUME = {166},
 PUBLISHER = {Princeton University Press},
   ADDRESS = {Princeton, NJ},
      YEAR = {2007},
     PAGES = {vi+208},
      ISBN = {978-0-691-12923-5; 0-691-12923-1},
   MRCLASS = {14E15},
  MRNUMBER = {2289519},
}

\bib{Kollar11}{article}{
   author={Koll{\'a}r, J{\'a}nos},
   title={Sources of log canonical centers},
   note={To appear in {\it Minimal models and extremal rays -- proceedings of the conference in honor of Shigefumi Mori's 60th birthday}, Advanced Studies in Pure Mathematics, Kinokuniya Publishing House, Tokyo, arXiv:1107.2863},
   date={2011}
}

\bib{Kollar11b}{article}{
   author={Koll{\'a}r, J{\'a}nos},
   title={New examples of terminal and log canonical singularities},
   note={arXiv:1107.2864},
   date={2011}
}

\bib{Kollar12}{article}{
   author={Koll{\'a}r, J{\'a}nos},
   title={Links of complex analytic singularities},
 journal={Surveys in Differential Geometry},
volume={18},
pages={157-193},
dol={http://dx.doi.org/10.4310/SDG.2013.v18.n1.a4},
   date={2013}
}

\bib{kk-singbook}{book}{
    AUTHOR = {Koll{\'a}r, J{\'a}nos},
     TITLE = {Singularities of the minimal model program},
    VOLUME={200},
        NOTE = {With the collaboration of S. Kov\'acs},
 PUBLISHER = {Cambridge University Press},
   ADDRESS = {Cambridge},
      YEAR = {2013},
   }

\bib{Kollar13}{article}{
   author = {Koll{\'a}r, J{\'a}nos},
    title = {Simple normal crossing varieties with prescribed dual complex},
  journal = {Algebraic Geometry},
  keywords = {Mathematics - Algebraic Geometry, 14B05, 14J17, 14M99},
     year = {2014},
    number = {1},
    pages={57-68},
     adsurl = {http://adsabs.harvard.edu/abs/2013arXiv1301.1089K},
  adsnote = {2013arXiv1301.1089K}
}




\bib{KS11}{article}{
   author={Kerz, Moritz},
   author={Saito, Shuji},
   title={Cohomological Hasse principle and resolution of quotient singularities},
      journal={New York J. of Math. (to appear)}
   note={arXiv:1111.7177},
   date={2011},
}

\bib{Kul77}{article}{
    AUTHOR = {Kulikov, Vik. S.},
     TITLE = {Degenerations of {$K3$} surfaces and {E}nriques surfaces},
   JOURNAL = {Izv. Akad. Nauk SSSR Ser. Mat.},
  FJOURNAL = {Izvestiya Akademii Nauk SSSR. Seriya Matematicheskaya},
    VOLUME = {41},
      YEAR = {1977},
    NUMBER = {5},
     PAGES = {1008--1042, 1199},
      ISSN = {0373-2436},
   MRCLASS = {14J25},
  MRNUMBER = {0506296 (58 \#22087b)},
MRREVIEWER = {I. Dolgachev},
}


\bib{LX11}{article}{
   author={Li, Chi},
   author={Xu, Chenyang},
   title={Special test configurations and K-stability of Fano varieties},
      journal={Ann. of Math. (to appear)}
   note={arXiv:1111.5398}
   date={2011}
}

\bib{OX12}{article}{
   author={Odaka, Yuji},
   author={Xu, Chenyang},
   title={Log-canonical models of singular pairs and its applications},
   journal={Math. Res. Lett.},
   volume={19}
   pages={325--334}
   number={2}
   date={2012}
}


\bib{pachner}{article}{
   author={Pachner, U.},
title={PL homeomorphic manifolds are equivalent by elementary shellings},
journal={Eur. J. Comb.},
 volume={12}
   pages={129--145}
   number={2}
   date={1991}
}

\bib{Payne11}{article}{
   author={Payne, Sam},
   title={Boundary complexes and weight filtrations},
   Journal={Mich. Math. J.}
  volume={62}
   pages={293--322}
   date={2013}
   }

\bib{par-sri}{article}{
    AUTHOR = {Parameswaran, A. J.},
  AUTHOR = {Srinivas, V.},
     TITLE = {A variant of the {N}oether-{L}efschetz theorem: some new
              examples of unique factorisation domains},
   JOURNAL = {J. Algebraic Geom.},
  FJOURNAL = {Journal of Algebraic Geometry},
    VOLUME = {3},
      YEAR = {1994},
    NUMBER = {1},
     PAGES = {81--115},
}




\bib{Sho92}{article}{
    AUTHOR = {Shokurov, V. V.},
     TITLE = {Three-dimensional log perestroikas},
   JOURNAL = {Izv. Ross. Akad. Nauk Ser. Mat.},
  FJOURNAL = {Rossi\u\i skaya Akademiya Nauk. Izvestiya. Seriya
              Matematicheskaya},
    VOLUME = {56},
      YEAR = {1992},
    NUMBER = {1},
     PAGES = {105--203},
      ISSN = {0373-2436},
   MRCLASS = {14E05 (14E35)},
  MRNUMBER = {1162635 (93j:14012)},
MRREVIEWER = {A. S. Tikhomirov},
}
\bib{Shokurov00}{article}{
   author={Shokurov, V. V.},
   title={Complements on surfaces},
   note={Algebraic geometry, 10},
   journal={J. Math. Sci. (New York)},
   volume={102},
   date={2000},
   number={2},
   pages={3876--3932},
}

\bib{Spanier}{book}{
author={Spanier, E.H.}
title={Algebraic topology}
PUBLISHER = {McGraw-Hill},
date={1966},
}

\bib{steenbrink}{article}{
    AUTHOR = {Steenbrink, J. H. M.},
     TITLE = {Mixed {H}odge structures associated with isolated
              singularities},
 BOOKTITLE = {Singularities, {P}art 2 ({A}rcata, {C}alif., 1981)},
    SERIES = {Proc. Sympos. Pure Math.},
    VOLUME = {40},
     PAGES = {513--536},
 PUBLISHER = {Amer. Math. Soc.},
   ADDRESS = {Providence, RI},
      YEAR = {1983},
   MRCLASS = {32G11 (14B05 14B07 14C30 14D05)},
  MRNUMBER = {713277 (85d:32044)},
MRREVIEWER = {Jerome William Hoffman},
}

\bib{Stepanov06}{article}{
   author={Stepanov, D. A.},
   title={A remark on the dual complex of a resolution of singularities},
   language={Russian},
   journal={Uspekhi Mat. Nauk},
   volume={61},
   date={2006},
   number={1(367)},
   pages={185--186},
   translation={
      journal={Russian Math. Surveys},
      volume={61},
      date={2006},
      number={1},
      pages={181--183},
   },
}
		
\bib{Stepanov07}{article}{
   author={Stepanov, D. A.},
   title={Dual complex of a resolution of terminal singularities},
   language={Russian},
   conference={
      title={Proceedings of the Fifth Kolmogorov Lectures (Russian)},
   },
   book={
      publisher={Yaroslav. Gos. Ped. Univ. im. K. D. Ushinskogo,
   Yaroslavl\cprime},
   },
   date={2007},
   pages={91--99},
}
		
\bib{Stepanov08}{article}{
   author={Stepanov, D. A.},
   title={A note on resolution of rational and hypersurface singularities},
   journal={Proc. Amer. Math. Soc.},
   volume={136},
   date={2008},
   number={8},
   pages={2647--2654},
}

\bib{szabo}{article}{
    AUTHOR = {Szab{\'o}, Endre},
     TITLE = {Divisorial log terminal singularities},
   JOURNAL = {J. Math. Sci. Univ. Tokyo},
  FJOURNAL = {The University of Tokyo. Journal of Mathematical Sciences},
    VOLUME = {1},
      YEAR = {1994},
    NUMBER = {3},
     PAGES = {631--639},
      ISSN = {1340-5705},
   MRCLASS = {14E15 (14E30)},
  MRNUMBER = {1322695 (96f:14019)},
MRREVIEWER = {Mark Gross},
}

\bib{Takayama03}{article}{
   author={Takayama, Shigeharu},
   title={Local simple connectedness of resolutions of log-terminal
   singularities},
   journal={Internat. J. Math.},
   volume={14},
   date={2003},
   number={8},
   pages={825--836},
}

\bib{Thuillier07}{article}{
   author={Thuillier, Amaury},
   title={G\'eom\'etrie toro\"\i dale et g\'eom\'etrie analytique non
   archim\'edienne. Application au type d'homotopie de certains sch\'emas
   formels},
   language={French, with English summary},
   journal={Manuscripta Math.},
   volume={123},
   date={2007},
   number={4},
   pages={381--451},
}



\bib{zee61}{article}{
AUTHOR = {Zeeman, E.C.},
TITLE = {On the dunce hat},
JOURNAL = {Topology},
volume={2}
 YEAR = {1961},
     PAGES = {341--358},
 }



\end{biblist}
\end{bibdiv}
\end{document}